\documentclass[a4paper,11pt]{amsart}
\usepackage{amsmath,amsthm,amssymb,amsfonts,enumerate,color}
\usepackage[pdftex]{graphicx}

\newcommand{\R}{\mathbb{R}}

\newcommand{\N}{\mathbb{N}}
\renewcommand{\S}{\mathcal{S}}
\newcommand{\vep}{\varepsilon}
\newcommand{\w}{\omega}
\newcommand{\supp}{\operatorname{supp}}

\newcommand{\dr}{(D_{\mathrm{right}})^{\alpha}}
\newcommand{\dl}{(D_{\mathrm{left}})^{\alpha}}
\newcommand{\norm}[1]{\left\Vert#1\right\Vert}
\newcommand{\abs}[1]{\left\vert#1\right\vert}

\renewcommand{\l}{\mathrm{left}}
\renewcommand{\r}{\mathrm{right}}
\def\({\left(}
\def\){\right)}

\newtheorem{thm}{Theorem}[section]
\newtheorem{prop}[thm]{Proposition}

\newtheorem{lem}[thm]{Lemma}
\theoremstyle{definition}
\newtheorem{defn}[thm]{Definition}
\newtheorem{rem}[thm]{Remark}

\numberwithin{equation}{section}

\allowdisplaybreaks

\author[P. R. Stinga]{Pablo Ra\'ul Stinga}
    \address{Department of Mathematics\\
    Iowa State University\\
    396 Carver Hall, Ames, IA 50011, USA}
    \email{stinga@iastate.edu,~maryo@iastate.edu}

\author[M. Vaughan]{Mary Vaughan}

\keywords{Marchaud fractional derivative, one-sided weights, weighted Sobolev space, maximal operator,
fractional Laplacian}

\subjclass[2010]{Primary: 26A24,26A33,46E35. Secondary: 35R11,42B25,46E30}

\thanks{Research supported by Simons Foundation grant 580911 (PRS).
The first author was partially supported by grant MTM2015-66157-C2-1-P (MINECO/FEDER) from Government of Spain}

\begin{document}

\title[Fractional derivatives, fractional Laplacians, and Sobolev spaces]{One-sided fractional derivatives, \\
fractional Laplacians, and weighted Sobolev spaces}

\begin{abstract}
We characterize one-sided weighted Sobolev spaces $W^{1,p}(\R,\w)$, where $\w$
is a one-sided Sawyer weight, in terms of a.e.~and weighted $L^p$ limits as $\alpha\to1^-$ of
Marchaud fractional derivatives of order $\alpha$.
Similar results for weighted Sobolev spaces $W^{2,p}(\R^n,\nu)$, where
$\nu$ is an $A_p$-Muckenhoupt weight, are proved in terms of limits as $s\to1^-$ of fractional Laplacians $(-\Delta)^s$.
These are Bourgain--Brezis--Mironescu-type characterizations for weighted Sobolev spaces.
We also complement their work by studying a.e.~and weighted $L^p$ limits as $\alpha,s\to0^+$.
\end{abstract}

\maketitle

\vspace{-.15in}

\section{Introduction and main results}

G.~Leibniz introduced the notation
$$\frac{d^n}{dt^n}u(t)$$
for derivatives of integer order $n\geq1$ of a function $u=u(t):\R\to\R$.
In 1695, G.~L'H\"opital posed Leibniz the question:
\begin{center}
\textit{What if $n=1/2$?}
\end{center}
Since then, many ``derivatives of fractional order'' have been defined. Historical names
are Lacroix, Fourier, Liouville, Riemann, Riesz, Weyl and, more recently, Chapman, Marchaud, Caputo, Jumarie,
Gr\"unwald and Letnikov, among others, see for instance \cite{Samko}. In our opinion, any reasonable definition of
derivative $D^\alpha$ of fractional order $0<\alpha<1$ should at least satisfy the relations
$D^\alpha[D^\beta u](t)=D^{\alpha+\beta} u(t)$,
$$\lim_{\alpha\to1^-}D^\alpha u(t)=u'(t)\quad\hbox{and}\quad\lim_{\alpha\to0^+}D^\alpha u(t)=u(t)$$
whenever $u$ is a sufficiently smooth function.

By looking at the various definitions of fractional derivatives \cite{Samko}, one notices
that most of them have a \textit{one-sided} nature. For example, the Marchaud left fractional derivative, given by
\begin{equation}\label{eq:Dleftalpha}
(D_\l)^\alpha u(t)=\frac{1}{\Gamma(-\alpha)}\int_{-\infty}^t\frac{u(\tau)-u(t)}{(t-\tau)^{1+\alpha}}\,d\tau
\end{equation}
where $\Gamma$ denotes the Gamma function,
takes into account the values of $u$ to the left of $t$ (the \textit{past}). Similarly, the Marchaud right fractional derivative
\begin{equation}\label{eq:Drightalpha}
(D_\r)^\alpha u(t)=\frac{1}{\Gamma(-\alpha)}\int_t^{\infty}\frac{u(\tau)-u(t)}{(\tau-t)^{1+\alpha}}\,d\tau
\end{equation}
looks at $u$ only to the right of $t$ (the \textit{future}).
These were first introduced by 
Andr\'e Marchaud in his 1927 dissertation \cite{Marchaud}
(see also, for example, \cite{Allen,Allen-Caffarelli-Vasseur1,Allen-Caffarelli-Vasseur2,Bernardis-et-al,Samko}
for theory and applications).
It is clear that if $u$ is a Schwartz class function, then
$$\lim_{\alpha\to1^-}(D_\l)^\alpha u(t)=u'(t)\qquad\hbox{and}\quad \lim_{\alpha\to0^+}(D_\l)^\alpha u(t)=u(t).$$

In this paper, we study characterizations of Sobolev spaces by limits of fractional derivatives 
in the almost everywhere and $L^p$ senses. Of course, an obvious
class of functions $u$ to work with is the classical Sobolev space $W^{1,p}(\R)$.
Instead, given the one-sided structure of fractional derivatives, we believe that a more natural,
general class of functions to consider is the weighted Sobolev space $W^{1,p}(\R,\w)$, $1\leq p<\infty$,
but where $\w$ is now a \textit{one-sided} Sawyer weight in $A_p^-(\R)$
(for left-sided fractional derivatives) or in $A_p^+(\R)$ (for right-sided fractional derivatives).
These spaces are defined as
$$W^{1,p}(\R,\w)=\big\{u\in L^p(\R,\w):u'\in L^p(\R,\w)\big\}$$
with the norm
$$\|u\|_{W^{1,p}(\R,\w)}^p=\|u\|_{L^p(\R,\w)}^p+\|u'\|_{L^p(\R,\w)}^p$$
for $1\leq p<\infty$. The Sawyer weights $\w\in A_p^-(\R)$ are the good weights 
for the original one-sided Hardy--Littlewood maximal function \cite[p.~92]{Hardy-Littlewood}:
$$M^-u(t)=\sup_{h>0}\frac{1}{h}\int_{t-h}^t|u(\tau)|\,d\tau.$$
Indeed, $M^-$ is bounded in $L^p(\R,\w)$ if and only if $\w\in A_p^-(\R)$, $1<p<\infty$,
see \cite{Sawyer}, and $M^-$ is bounded from $L^1(\R,\w)$ into weak-$L^1(\R,\w)$ if and only if
$\w\in A_1^-(\R)$, see \cite{Martin-Reyes-et-al}.
 It is clear that $A_p^-(\R)$ is a larger family than the classical
class of Muckenhoupt weights $A_p(\R)$. 
In particular, any decreasing function is in $A_p^-(\R)$, but there
are decreasing functions that are not in $A_p(\R)$.
For instance, $\w(t) = e^{-t}$ belongs to $A_p^-(\R)$ but not to $A_p(\R)$ because it is not a doubling weight.
Similar considerations hold for right-sided weights in $A_p^+(\R)$.
See Section \ref{Section:PreliminariesDerivatives} for more details.

We find appropriate one-sided distributional spaces in which fractional derivatives have sense.
Then we show that in such a setting one can always define $(D_\l)^\alpha u$
as a distribution for any function $u\in L^p(\R,\w)$, $\w\in A_p^-(\R)$. 
It turns out then that our weighted Sobolev spaces can be characterized by
limits of one-sided left fractional derivatives.

\begin{thm}[$W^{1,p}(\R,\w)$ and limits of left fractional derivatives]\label{thm:derivative1}
Let $u \in L^p(\R,\w)$, where $\w \in A_p^-(\R)$, for $1\leq p<\infty$.
\begin{enumerate}[$(a)$]
\item If $u \in W^{1,p}(\R,\w)$, then the distribution $(D_\l)^\alpha u$ coincides with a function in $L^p(\R,\w)$ and 
\begin{equation}\label{eq:formulaDleftalpha}
\dl u(t) = \frac{1}{\Gamma(-\alpha)}\int_{-\infty}^t\frac{u(\tau)-u(t)}{(t-\tau)^{1+\alpha}}\,d\tau \quad\hbox{for a.e.}~t\in\R
\end{equation}
with
\begin{equation}\label{eq:LpestimateDleftalpha}
\|\dl u\|_{L^p(\R,\w)} \leq C_{p,\w}\big( \|u\|_{L^p(\R,\w)} + \|u'\|_{L^p(\R,\w)}\big)
\end{equation}
for some constant $C_{p,\w}>0$. Moreover,
\begin{equation}\label{eq:limitalpha}
\lim_{\alpha\to1^-}\dl u=u'\quad\hbox{in}~L^p(\R,\w)~\hbox{and a.e.~in}~\R
\end{equation}
and
\begin{equation}\label{eq:limitalpha2}
\lim_{\alpha\to0^+}\dl u=u\quad\hbox{a.e.~in}~\R.
\end{equation}
Furthermore, the limit in \eqref{eq:limitalpha2} holds also in $L^p(\R,\w)$ when $1 < p < \infty$, and in $L^{1}(\R, \w)$
when $p=1$ and $M^-u,M^-u'\in L^1(\R,\w)$.
\item Conversely, suppose that $\dl u \in L^p(\R,\w)$ and that $\dl u$ converges in $L^p(\R,\w)$
as $\alpha \to 1^-$. Then $u \in W^{1,p}(\R,\w)$ and \eqref{eq:limitalpha} holds.
\item Alternatively, suppose that $\dl u \in L^p(\R,\w)$ and that $\dl u$ converges in $L^p(\R,\w)$
as $\alpha\to 0^+$. Then \eqref{eq:limitalpha2} holds and, as a consequence, $\dl u\to u$ in $L^p(\R,\w)$ as $\alpha\to0^+$.
\end{enumerate}
\end{thm}

Though we established Theorem \ref{thm:derivative1} for the left fractional derivative,
all the arguments carry on by replacing $D_\l$ by $D_\r$ and $A_p^-(\R)$ by $A_p^+(\R)$. Hence, for the rest of the paper,
we will only consider the case of $D_\l$ and left-sided Sawyer weights.

The one-sided $L^p(\R,\w)$ spaces, with $\w\in A_p^-(\R)$, are also natural for the Marchaud
left fractional derivative
in the sense of the Fundamental Theorem of Fractional Calculus. Indeed, let $u\in L^p(\R,\w)$ and
consider the left-sided Weyl fractional integral \cite{Samko}
$$(D_\l)^{-\alpha}u(t)=\frac{1}{\Gamma(\alpha)}\int_{-\infty}^t\frac{u(\tau)}{(t-\tau)^{1-\alpha}}\,d\tau.$$
It was proved in \cite{Bernardis-et-al}
that $(D_\l)^\alpha(D_\l)^{-\alpha}u(t)=u(t)$ in $L^p(\R,\w)$ and for a.e.~$t\in\R$, for any $0<\alpha<1$.
Our Theorem \ref{thm:derivative1} complements this result.

The second question we address in this paper is the almost everywhere and $L^p$ characterization of weighted
Sobolev spaces by the limits
$$\lim_{s\to1^-}(-\Delta)^su=-\Delta u\quad\hbox{and}\quad\lim_{s\to0^+}(-\Delta)^su=u$$
where $(-\Delta)^s$ is the fractional Laplacian of order $0<s<1$ on $\R^n$, $n\geq1$.
Both limits hold whenever $u$ is a Schwartz class function.
Up to the best of our knowledge,
they have not been studied for the case of weighted $L^p$ spaces. We will consider the weighted
Sobolev space $W^{2,p}(\R^n,\nu)$ defined by
$$W^{2,p}(\R^n,\nu)=\big\{u\in L^p(\R^n,\nu):\nabla u,D^2u\in L^p(\R^n,\nu)\big\}$$
with the norm
$$\|u\|_{W^{2,p}(\R^n,\nu)}^p=\|u\|_{L^p(\R^n,\nu)}^p+\|\nabla u\|_{L^p(\R^n,\nu)}^p+\|D^2u\|_{L^p(\R^n,\nu)}^p$$
where $\nu$ is a weight in the Muckenhoupt class $A_p(\R^n)$
(see Section \ref{Section:prelim Laplacian}), for
$1\leq p<\infty$.
We recall that the $A_p(\R^n)$ Muckenhoupt weights are the good weights for the classical Hardy--Littlewood
maximal function $M$ on $\R^n$. In the following statement, $\{e^{t\Delta}\}_{t\geq0}$ denotes the 
heat semigroup generated by the Laplacian on $\R^n$.

\begin{thm}[$W^{2,p}(\R^n,\nu)$ and limits of fractional Laplacians]\label{thm:laplacian1}
Let $u\in L^p(\R^n,\nu)$, where $\nu \in A_p(\R^n)$, for $1\leq p<\infty$.
\begin{enumerate}[$(a)$]
\item If $u \in W^{2,p}(\R^n,\nu)$, then the distribution $(-\Delta)^s u$ coincides with a function in $L^p(\R^n,\nu)$ and 
\begin{equation} \label{eq:laplacian semigroup}
(-\Delta)^su(x) = \frac{1}{\Gamma(-s)} \int_0^{\infty} \( e^{t \Delta} u(x) - u(x)\) \frac{dt}{t^{1+s}} \quad
\hbox{for a.e.}~x\in\R^n.
\end{equation}
In addition,
\begin{equation} \label{eq:laplacian pw}
(-\Delta)^su(x)= c_{n,s} \lim_{\varepsilon \to 0} \int_{\abs{x-y}> \varepsilon} \frac{u(x) -u(y)}{\abs{x-y}^{n+2s}}\, dy
\quad\hbox{for a.e.}~x\in\R^n~\hbox{and in}~L^p(\R^n,\nu)
\end{equation}
with
\begin{equation} \label{eq:lpestimatelaplacian}
\norm{(-\Delta)^s u}_{L^p(\R^n,\nu)} \leq C_{n,p,\nu} \big(\norm{u}_{L^p(\R^n,\nu)} + \norm{\Delta u}_{L^p(\R^n,\nu)} \big)
\end{equation}
for some constant $C_{n,p,\nu}>0$. Moreover,
\begin{equation}\label{eq:limits}
\lim_{s\to1^-}(-\Delta)^su=-\Delta u\quad\hbox{in}~L^p(\R^n,\nu)~\hbox{and a.e.~in}~\R^n
\end{equation}
and
\begin{equation}\label{eq:limits2}
\lim_{s\to0^+}(-\Delta)^su=u\quad\hbox{a.e.~in}~\R^n.
\end{equation}
Furthermore, the limit in \eqref{eq:limits2} holds also in $L^p(\R^n,\nu)$ when $1 < p < \infty$, and in
$L^{1}(\R^n, \nu)$ when $p=1$ and $Mu,M(D^2u)\in L^1(\R^n,\nu)$.
\item Conversely, suppose that $(-\Delta)^s u \in L^p(\R^n,\nu)$ and that $(-\Delta)^s u$ converges in $L^p(\R^n,\nu)$ as $s \to 1^-$.
If $1<p<\infty$ then $u \in W^{2,p}(\R^n,\nu)$ and \eqref{eq:limits} holds. If $p=1$, then $D^2u\in$ \emph{weak}-$L^1(\R^n,\nu)$.
\item Alternatively, suppose that $(-\Delta)^su \in L^p(\R^n,\nu)$ and that $(-\Delta)^su$ converges in $L^p(\R^n,\nu)$
as $s\to 0^+$. Then \eqref{eq:limits2} holds and, as a consequence, $(-\Delta)^su\to u$ in $L^p(\R^n,\nu)$ as $s\to0^+$.
\end{enumerate}
\end{thm}

Our Theorems \ref{thm:derivative1} and \ref{thm:laplacian1} are rather nontrivial, nonisotropic weighted versions of the
famous results by Bourgain--Brezis--Mironescu \cite{Bourgain-Brezis-Mironescu}. Indeed,
\cite{Bourgain-Brezis-Mironescu} gives a characterization of $W^{1,p}(\Omega)$, $\Omega\subseteq\R^n$,
in terms of the limit as $s\to1^-$ of fractional Gagliardo seminorms, namely, the seminorms
of the fractional Sobolev spaces $W^{s,p}(\Omega)$. 
Other authors have considered similar questions for abstract versions of such seminorms, see for example
\cite{DiMarino-Squassina,Vincent}. In particular, they apply to Ahlfors-regular metric spaces.
On the other hand,
a weighted Gagliardo-type fractional seminorm with power weights was defined in \cite{Dipierro-Valdinoci}.
Nevertheless, neither are our weighted spaces Ahlfors-regular nor do our seminorms $\|(D_\l)^\alpha u\|_{L^p(\R,\w)}$
and $\|(-\Delta)^su\|_{L^p(\R^n,\nu)}$ correspond to those in \cite{Dipierro-Valdinoci}, even for power weights.
An added difficulty we need to overcome in our case is the lack of translation invariance of the one-sided
weighted $L^p(\R,\w)$ spaces.
Moreover, since constants are not in our weighted Sobolev spaces,
we are able to complement \cite{Bourgain-Brezis-Mironescu} by studying limits as $\alpha,s\to0^+$.

In general, statements involving a.e.~convergence are proved by considering the underlying maximal operators,
see, for example, \cite[Chapter~2]{Duo}. One of the novelties of our paper is that we are able to deduce the
pointwise inequalities
\begin{equation}\label{eq:Ma}
\sup_{0<\alpha<1}|(D_\l)^\alpha u(t)|\leq C\big(M^-(u')(t)+M^-u(t)\big)\quad\hbox{for any}~u\in W^{1,p}(\R,\omega)
\end{equation}
and
\begin{equation}\label{eq:Ms}
\sup_{0<s<1}|(-\Delta)^su(x)|\leq C_n\big(M(D^2u)(x)+Mu(x)\big)\quad\hbox{for any}~u\in W^{2,p}(\R^n,\nu),
\end{equation}
see Theorems \ref{Prop:lp max} and \ref{prop:Mbound}, respectively. The constant $C>0$ in \eqref{eq:Ma}
is universal while $C_n>0$ in \eqref{eq:Ms} depends only on dimension.
Notice that the maximal operators are taken with respect to the orders of the fractional derivative
and the fractional Laplacian, respectively. We believe these estimates are of independent interest.

The paper is organized as follows. Section \ref{Section:PreliminariesDerivatives} contains preliminary
results on one-sided Sawyer weights, the new distributional setting for one-sided fractional derivatives,
and the proof of the maximal estimate \eqref{eq:Ma}.
Theorem \ref{thm:derivative1} is proved in Section \ref{Section:proofs derivatives}.
The fractional Laplacian in weighted Lebesgue spaces is studied in detail in Section \ref{Section:prelim Laplacian},
where we also show the maximal estimate \eqref{eq:Ms}.
Finally, Section \ref{Section:proofs Laplacians} contains the proof of Theorem \ref{thm:laplacian1}.

Along the paper, we denote by $\mathcal{S}(\R^n)$ the class of Schwartz functions on $\R^n$.
We always take $0<\alpha,s<1$. We will use the following inequality:
for any fixed $\rho>0$ there exists $C_\rho>0$ such that, for every $r>0$,
\begin{equation}\label{eq:exponential}
e^{-r}r^{\rho} \leq C_{\rho} e^{-r/2}.
\end{equation}

For a measure space $(X, \mu)$, we define the space weak-$L^1(X, \mu)$ as the set of measurable functions
$u:X\to\R$ such that the quasi-norm $\norm{\cdot}_{\textrm{weak-}L^1(X, \mu)}$, defined by
$$\norm{u}_{\textrm{weak-}L^1(X, \mu)} = \sup_{\lambda >0}  \lambda \mu( \{ x \in X : \abs{u(x)} > \lambda\}),$$
is finite.

\section{Fractional derivatives and one-sided spaces}\label{Section:PreliminariesDerivatives}

Let $u=u(t)\in\S(\R)$ and define
$$D_\l u(t)=\lim_{\tau\to0^+}\frac{u(t)-u(t-\tau)}{\tau}\quad\hbox{and}\quad
D_\r u(t)=\lim_{\tau\to0^+}\frac{u(t)-u(t+\tau)}{\tau}.$$
Observe that $D_\l u=-D_\r u=u'$. From the Fourier transform identities
$$\widehat{D_\l u}(\xi)=(i\xi)\widehat{u}(\xi)\quad\hbox{and}\quad\widehat{D_\r u}(\xi)=(-i\xi)\widehat{u}(\xi),$$
one can define 
\begin{equation}\label{eq:Fourier}
\widehat{(D_\l)^\alpha u}(\xi)=(i\xi)^\alpha\widehat{u}(\xi)\quad\hbox{and}\quad
\widehat{(D_\r)^\alpha u}(\xi)=(-i\xi)^\alpha\widehat{u}(\xi).
\end{equation}
Using the semigroup of translations, it is shown in \cite{Bernardis-et-al}, see also \cite{Samko}, that $(D_\l)^\alpha u(t)$
and $(D_\r)^\alpha u(t)$ are given by the pointwise formulas in \eqref{eq:Dleftalpha}
and \eqref{eq:Drightalpha}, respectively.

\subsection{Distributional setting}

If $u,\varphi\in\S(\R)$, then
$$\int_{-\infty}^\infty \dl u \, \varphi \, dt = \int_{-\infty}^\infty u \, \dr \varphi \, dt.$$
We will use this identity to define $\dl u$ in the sense of distributions. Notice that
if $u\in\S'(\R)$, then a natural definition would be
$$(\dl u)(\varphi) = u \( \dr \varphi \).$$
Nevertheless, it is straightforward from \eqref{eq:Fourier} to see that, in general,
$\dr \varphi \notin \S(\R)$, so we need to consider a different space of test functions and distributions.

We define the class
$$\S_-= \left\{ \varphi \in \S(\R) : \supp \varphi \subset (-\infty,A],~\hbox{for some}~A\in\R\right\}.$$
We denote by $\S_-^{\alpha}$ the set of functions
$$\varphi \in C^{\infty} (\R)\,\,\,\hbox{such that}~\supp\varphi \subset (-\infty, A]~\hbox{and}~\Big|\frac{d^k}{dt^k}\varphi(t)\Big|
\leq \frac{C}{1 + \abs{t}^{1+\alpha}}$$
for all $k\geq0$, for some $A\in\R$ and $C>0$.

\begin{lem}
If $\varphi\in\S_-$ then $(D_\r)^\alpha\varphi\in\S_-^{\alpha}$.
\end{lem}

\begin{proof}
Clearly, if $\varphi\in\S_-$ with $\supp\varphi\in(-\infty,A]$, then $(D_\r)^\alpha\varphi$ 
also has support in $(-\infty,A]$, see \eqref{eq:Drightalpha}.
Since $\dr\frac{d^k}{dt^k} \varphi = \frac{d^k}{dt^k}\dr \varphi$, we know $\dr \varphi \in C^{\infty}(\R)$ and only need to estimate $\dr \varphi$. 
If $-1<t <A$, the estimate holds because $\varphi$ is smooth and bounded. If $t>A$, then the estimate holds trivially because $(D_\r)^\alpha\varphi(t) = 0$. Suppose $-\infty < t < -1$ and write 
$$\int_{t}^\infty \frac{|\varphi(\tau) - \varphi(t)|}{|\tau-t|^{1+\alpha}} \, d \tau 
	=\int_{t}^{t/2} \frac{|\varphi(\tau) - \varphi(t)|}{|\tau-t|^{1+\alpha}} \, d \tau + \int_{t/2}^\infty \frac{|\varphi(\tau) - \varphi(t)|}{|\tau-t|^{1+\alpha}} \, d \tau=I+II. $$
For $I$, note that
\begin{align*}
|\varphi(\tau) - \varphi(t)|
	&\leq |\varphi'(\xi)| \, |\tau-t |\\
	&= |\varphi'(\xi)| (1+|\xi|)^3 \frac{ (\tau - t )}{(1+|\xi|)^3} 
	\leq C_\varphi \frac{ (\tau-t)}{(1+|\xi|)^3}
\end{align*}
where $\xi$ is some point in between $t$ and $\tau$. Hence,
$$I \leq \frac{C}{|t|^{3}} \int_{t}^{t/2} \frac{1}{(\tau -t)^{\alpha}} \, d \tau = \frac{C}{|t|^{2+\alpha}}\leq \frac{C}{1+|t|^{1+\alpha}}.$$
On the other hand, if $\tau>t/2$, then $\tau-t > -t/2>0$ and
\begin{align*}
II	&\leq \int_{t/2}^{\infty} \frac{|\varphi(\tau)|}{(\tau-t)^{1+\alpha}} \, d \tau 
		+|\varphi(t)|\int_{t/2}^{\infty} \frac{1}{(\tau-t)^{1+\alpha}} \, d \tau \\
	&\leq \frac{C}{|t|^{1+\alpha}} \norm{\varphi}_{L^1(\R)}  
		+\frac{C}{|t|^{1+\alpha}} |t\varphi(t)| 
	\leq \frac{C}{1+|t|^{1+\alpha}}.
\end{align*}
Collecting all the terms, we get
$$| \dr \varphi(t)| \leq C\int_{t}^\infty \frac{|\varphi(\tau) - \varphi(t)|}{|\tau-t|^{1+\alpha}} \, d \tau  \leq \frac{C}{1+\abs{t}^{1+\alpha}}$$
for all $t\in\R$. Thus, $(D_\r)^\alpha\varphi\in\S_-^\alpha$.
\end{proof}

We endow $\S_-$ and $\S_-^{\alpha}$ with the families of seminorms
$$\rho_{-}^{\ell,k}(\varphi)= \sup_{t\in\R} \abs{t}^{\ell}\Big|\frac{d^k}{dt^k}\varphi(t)\Big|\quad\hbox{for}~\ell,k\geq0,$$
and
$$\rho_-^{\alpha,k}(\varphi) = \sup_{t\in\R}(1 + \abs{t}^{1+\alpha})\Big|\frac{d^k}{dt^k}\varphi(t)\Big|\quad\hbox{for}~k\geq0,$$
respectively. Let us denote by $(\S_-)'$ and $(\S_-^{\alpha})'$
the corresponding dual spaces of $\S_-$ and $\S_-^{\alpha}$.
Notice that $\S_- \subset \S_-^{\alpha}$, so that $(\S_-^{\alpha})' \subset (\S_-)'$.
It turns out that $(\S_-^{\alpha})'$ is the appropriate class of distributions to extend the definition of the left fractional derivative.

\begin{defn}\label{definition}
For $u \in (\S_-^{\alpha})'$, we define $\dl u$ as the distribution in $(\S_-)'$ given by
$$(\dl u)(\varphi) = u ( \dr \varphi )\quad\hbox{for any}~\varphi\in \S_-.$$
\end{defn}

Consider next the class of functions given by
$$L_-^{\alpha} = \left\{ u \in L^1_{\mathrm{loc}}(\R) :
\int_{-\infty}^A \frac{\abs{u(\tau)}}{1 + \abs{\tau}^{1+\alpha}} \, d\tau < \infty,~\hbox{for any}~A\in\R \right\}.$$
We use the notation
$$\norm{u}_{A} = \int_{-\infty}^A \frac{|u(\tau)|}{1+ |\tau|^{1+\alpha}} \, d \tau\quad\hbox{for}~A\in\R.$$
Any function $u\in L_-^\alpha$ defines a distribution in $(\S_-^{\alpha})'$ in the usual way,
so that $(D_\l)^\alpha u$ is well defined as an object in $(\S_-)'$.
The following result is proved similarly as in the case of the fractional Laplacian, see Silvestre \cite{Silvestre}, 
so the details are omitted.

\begin{prop}
Let $u \in L_-^{\alpha}$. Assume that $u \in C^{\alpha + \varepsilon}(\mathcal{I})$ for some $\varepsilon >0$
and some open set $\mathcal{I}\subset\R$. Then $\dl u\in C(\mathcal{I})$ and
$$\dl u(t) = \frac{1}{\Gamma(-\alpha)} \int_{-\infty}^{t} \frac{u(\tau) - u(t)}{(t-\tau)^{1 + \alpha}} \, d \tau \quad \text{ for all } t \in\mathcal{I}.$$
\end{prop}

\begin{rem}
We have found that the one-sided class $L_-^\alpha$ is the appropriate space of locally integrable functions
to define the left fractional derivative.
This is a refinement with respect to the distributional definition presented in
\cite[Remark~2.6]{Bernardis-et-al}, which was two-sided in nature.
\end{rem}

\subsection{One-sided weighted spaces}

A nonnegative, locally integrable function $\w=\w(\tau)$ defined on $\R$
is in the left-sided Sawyer class $A_p^-(\R)$, for $1 < p <\infty$,
 if there exists $C>0$ such that
$$
\( \frac{1}{h} \int_{a}^{a+h} \w\,d\tau\)^{1/p} \( \frac{1}{h} \int_{a-h}^a \w^{1-p'}\,d\tau \)^{1/p'} \leq C
$$
for all $a\in\R$ and $h >0$, where $1/p +1/p' = 1$.
We then write $\w \in A_p^-(\R)$. 
By re-orienting the real line, one may similarly define the right-sided $A_p^+(\R)$-condition:
a weight $\tilde{\w}$ belongs to $A_p^+(\R)$ if there is a constant $C>0$ such that
$$\( \frac{1}{h} \int_{a-h}^{a} \tilde{\w}\,d\tau\)^{1/p} \( \frac{1}{h} \int_{a}^{a+h} \tilde{\w}^{1-p'}\,d\tau \)^{1/p'} \leq C$$
for all $a\in\R$ and $h >0$. In this way, $\w \in A_p^-(\R)$ if and only if $\w^{1-p'} \in A_{p'}^+(\R)$.

From the definition, one should note that, for $\w \in A_p^-(\R)$, there exist $-\infty \leq a < b \leq \infty$ such that $\w = \infty$ in $(-\infty,a)$, $0 < \w < \infty$ in $(a,b)$, $\w = 0$ in $(b, \infty)$, and $\w \in L^1_{\mathrm{loc}}((a,b))$. For simplicity
and without loss of generality, we will assume $(a,b) = \R$, so that $0 < \w < \infty$ in $\R$.

The one-sided Hardy--Littlewood maximal functions $M^-$ and $M^+$ are defined by
$$M^- u(t) = \sup_{h >0} \frac{1}{h} \int_{t-h}^{t} \abs{u(\tau)} \, d \tau
\quad \hbox{and} \quad M^+ u(t) = \sup_{h >0} \frac{1}{h} \int_{t}^{t+h} \abs{u(\tau)} \, d \tau$$
respectively. If $1 < p < \infty$, then $M^{\pm}$ is bounded on $L^p(\R,\w)$ if and only if $\w \in A_p^{\pm}(\R)$, see \cite{Sawyer}. 
When $p=1$, $M^\pm$ is bounded from $L^1(\R,\w)$ into weak-$L^1(\R,\w)$ if and only if $\w \in A_1^\pm(\R)$, namely,
there exists $C>0$ such that
$$M^\mp\w(t) \leq C \w(t) \quad\hbox{for a.e.}~t\in\R$$
see \cite{Martin-Reyes-et-al}. 
We refer to \cite{Lorente,MartinBMO,Cabrelli-Torrea,Martin-Reyes-et-al,Canadian,Sawyer}
for these and more properties of one-sided weights. For a measurable set $E\subset\R$, we denote
$$\w(E)=\int_E\w\,d\tau.$$
An important property that we will use is the following.

\begin{lem}[{See \cite[Theorem~3]{Lorente}}]\label{LemmaLorente}
Let $\eta=\eta(t)\geq0$ be a integrable function with support in $[0,\infty)$ and nonincreasing in $[0,\infty)$. Then, for any
measurable function $u:\R\to\R$ and for almost all $t\in\R$, we have
$$|u\ast\eta(t)|\leq M^-u(t)\int_0^\infty\eta(\tau)\,d\tau.$$
By changing the orientation of the real line, the analogue conclusion holds for nondecreasing $\eta$ supported in $(-\infty,0]$
with $M^+$ in place of $M^-$.
\end{lem}

\begin{lem}[See {\cite[Theorem~1]{Canadian}}]\label{lem:WRHI}
If $\w \in A_p^-(\R)$, $1 \leq p < \infty$, then there exist $C, \delta >0$ such that
$$\frac{\w(E)}{\w((a,c))} \leq C \( \frac{\abs{E}}{b-a}\)^{\delta} $$
for all $a<b<c$ and all measurable subsets $E \subset (b,c)$.
\end{lem}

\begin{lem}\label{lem:observation}
If $\w \in A_1^-(\R)$, then there is a constant $C>0$ such that, for any $0 < a < b$, 
$$ \frac{\w((-a,-a+(b-a)))}{2(b-a)} \leq C \inf_ {-b <t <-a}\w(t).$$
\end{lem}

\begin{proof}
Let $t \in (-b,-a)$. Since $(-a, -a+(b-a)) \subset (t, t+2(b-a))$, then, by the $A_1^-(\R)$-condition, we get
\begin{align*}
C\w(t)&\geq M^+\w(t) 
	\geq \frac{1}{2(b-a)} \int_t^{t+2(b-a)} \w(\tau) \, d\tau \\
	&\geq \frac{1}{2(b-a)} \int_{-a}^{-a+(b-a)} \w(\tau) \, d\tau
	=  \frac{\w((-a,-a+(b-a)))}{2(b-a)}
\end{align*}
for almost every $t\in\R$.
\end{proof}

The following result says that $(D_\l)^\alpha u$ is well defined as a distribution in $(\S_-)'$ whenever $u\in L^p(\R,\w)$,
for $\w\in A_p^-(\R)$, $1\leq p<\infty$.

\begin{prop}\label{prop:A}
If $\w \in A_p^-(\R)$, $1\leq p<\infty$, then $L^p(\R,\w) \subset L_-^{\alpha}$, $\alpha\geq0$, and, for any $A\in\R$,
there is a constant $C=C_{A,\w,p}>0$ such that
$$\|u\|_A\leq C\|u\|_{L^p(\R,\w)}.$$
In particular, $L^p(\R,\w)\subset L^1_{\mathrm{loc}}(\R)$.
\end{prop}

\begin{proof}
Let $u \in L^p(\R,\w)$ and fix any $A\in\R$.

We first let $1 < p < \infty$. By H\"older's inequality,
\begin{align*}
\norm{u}_{A}
	&= \int_{-\infty}^A \frac{\abs{u(\tau)}}{1 + \abs{\tau}^{1+\alpha}} \, d\tau
	= \int_{- \infty}^A \abs{u(\tau)} \w(\tau)^{1/p} \, \frac{\w(\tau)^{-1/p}}{1 + \abs{\tau}^{1 + \alpha}} \, d \tau \\
	&\leq \norm{u}_{L^p(\R,\w)}\( \int_{-\infty}^A \frac{\w(\tau)^{-p'/p}}{( 1 + \abs{\tau})^{p'}}\, d \tau \)^{1/p'}=\norm{u}_{L^p(\R,\w)}\cdot (I_A)^{1/p'}.
\end{align*}
Observe that $\tilde{\w}(\tau)= \w(\tau)^{-p'/p} = \w(\tau)^{1-p'} \in A_{p'}^+(\R)$. 
To conclude, it is enough to recall that
$$I=\int_{-\infty}^0 \frac{\tilde{\w}(\tau)}{(1 + \abs{\tau})^{p'}} \, d\tau < \infty,$$
see \cite[Lemma~4]{MartinBMO}.

Now let $p=1$. For convenience with the notation, we let $A =0$ (the general case follows the same lines).
First observe that, by the $A_1^-(\R)$-condition,
\begin{align*}
 \int_{-1}^0 \frac{\abs{u(\tau)}}{1 + \abs{\tau}^{1+\alpha}} \, d \tau 
	&\leq \int_{- 1}^0 \abs{u(\tau)} \w(\tau) \w(\tau)^{-1} \, d \tau \\
	&\leq  \norm{u}_{L^1(\R,\w)} \sup_{t \in (-1,0)}\w(t)^{-1}\\
	&=  \norm{u}_{L^1(\w)}\( \inf_{t \in (-1,0)} \w(t)\)^{-1}\leq \norm{u}_{L^1(\w)} \frac{C}{\w((-1,0))}<\infty.
\end{align*}
On the other hand, by Lemma \ref{lem:observation},
\begin{align*}
 \int_{-\infty}^{-1}\frac{\abs{u(\tau)}}{1 + \abs{\tau}^{1+\alpha}} \, d \tau 
	&\leq \sum_{k=0}^{\infty} \int_{-2^{k+1}}^{-2^k} \frac{\abs{u(\tau)}}{\abs{\tau}} \, d\tau\\
	&\leq \sum_{k=0}^{\infty} \frac{1}{2^k} \int_{-2^{k+1}}^{-2^k} \abs{u(\tau)}\w(\tau) \w(\tau)^{-1} \, d\tau\\
	&\leq  \norm{u}_{L^1(\R,\w)}\sum_{k=0}^{\infty} \frac{1}{2^k} \( \inf_{-2^{k+1} < t < -2^k} \w(t)\)^{-1} \\
	&\leq  C\norm{u}_{L^1(\R,\w)}\sum_{k=0}^{\infty} \frac{1}{\w((-2^k, 0))}.
\end{align*}
Lemma \ref{lem:WRHI} implies that there exist $C,\delta>0$ such that
$$\frac{\w((-1,0))}{\w((-2^{k},0))} \leq C \( \frac{1}{2^{k}} \)^{\delta}.$$
Whence,
$$ \int_{-\infty}^{-1}\frac{\abs{u(\tau)}}{1 + \abs{\tau}^{1+\alpha}} \, d \tau \leq 
\frac{C}{\w((-1,0))}\norm{u}_{L^1(\R,\w)}\sum_{k=0}^{\infty} \( \frac{1}{2^{k}} \)^{\delta}<\infty.$$
Thus, $u\in L_-^\alpha$ with the corresponding estimate.
\end{proof}

\subsection{Density of smooth functions in $W^{1,p}(\R,\w)$}

The proof of the following statement is similar to that of Lorente \cite[Theorem 3]{Lorente}. 
Indeed, the idea is to bound $\psi \in C^{\infty}_c([0,\infty))$ by a measurable function $\eta$ supported in $[0, \infty)$
which is nonincreasing in $[0, \infty)$, and follow the steps of the proof in \cite{Lorente}.

\begin{prop}\label{prop:density}
Let $\w \in A_p^-(\R)$ and $u \in L^p(\R, \w)$ for $1 \leq p < \infty$. Let $\psi \in C_c^{\infty}([0,\infty))$ such that $\displaystyle\int_0^\infty
\psi\,dt = 1$. Define $\psi_{\varepsilon}(t) = \frac{1}{\varepsilon} \psi \( \frac{t}{\varepsilon}\)$. Then the following hold.
\begin{enumerate}[$(1)$]
\item $\abs{u*\psi_{\varepsilon}(t)} \leq CM^-u(t)$  for almost every $t \in \R$.
\item $\norm{u*\psi_{\varepsilon}}_{L^p(\R, \w)} \leq C \norm{u}_{L^p(\R, \w)}$.
\item $\lim_{\varepsilon \to 0^+} u*\psi_{\varepsilon}(t) = u(t)$ for almost every $t \in \R$.
\item $\lim_{\varepsilon \to 0^+}\norm{u*\psi_{\varepsilon} - u}_{L^p(\R, \w)} = 0$.
\end{enumerate}
\end{prop}

It follows that $C^\infty(\R)\cap L^p(\R,\w)$ and $C^{\infty}_c(\R)$ are dense in $L^p(\R, \w)$ for $\w \in A_p^-(\R)$, $1\leq p < \infty$.
Additionally, notice that if $\psi$ is as in Proposition \ref{prop:density} and $u\in W^{1,p}(\R,\w)$, then
$$(u*\psi_{\varepsilon})'(t) = \int_{-\infty}^{\infty} u'(\tau)\psi_{\varepsilon}(t-\tau) \, d \tau = (u' *\psi_{\varepsilon})(t).$$
Hence $u*\psi_{\varepsilon} \to u$ as $\varepsilon \to 0^+$ in $W^{1,p}(\R, \w)$, so that $C^{\infty}(\R) \cap W^{1,p}(\R, \w)$ and $C^{\infty}_c(\R)$ are dense in $W^{1,p}(\R, \w)$ for $\w \in A_p^-(\R)$, $1\leq p < \infty$.

\subsection{The maximal estimate \eqref{eq:Ma}}

\begin{thm} \label{Prop:lp max} 
There exists a universal constant $C>0$ such that for any $u \in W^{1,p}(\R,\w)$, $\w \in A_p^-(\R)$, $1 \leq p < \infty$, we have
$$\sup_{0<\alpha<1}
\abs{\frac{1}{\Gamma(-\alpha)} \int_{0}^{\infty} \( u(t-\tau) - u(t)\) \frac{d \tau}{\tau^{1 + \alpha}}} \leq C \( M^-(u')(t) +M^-u(t) \)$$
for~a.e. $t \in \R$.
\end{thm}

\begin{proof}
We begin by writing
\begin{equation}\label{eq:splitting}
 I_\alpha + II_\alpha :=  \frac{1}{\Gamma(-\alpha)} \int_0^{1} \( u(t-\tau) - u(t)\) \frac{d \tau}{\tau^{1 + \alpha}} +  \frac{1}{\Gamma(-\alpha)} \int_1^{\infty} \( u(t-\tau) - u(t)\) \frac{d \tau}{\tau^{1 + \alpha}}.
 \end{equation}
To study $I_\alpha$, notice that
\begin{equation}\label{eq:Taylor's Remainder}
\begin{aligned}
\int_0^1 \abs{u(t-\tau) - u(t)} \frac{d \tau}{\tau^{1 + \alpha}}
	&\leq \int_0^1 \tau \int_0^1 \abs{ u'(t-r\tau)} dr \frac{d \tau}{\tau^{1 + \alpha}}\\
	&= \int_0^1 \int_0^1 \abs{u'( t- r\tau)} \frac{d\tau}{\tau^{\alpha}} \, dr \\
	&=   \int_0^1\(  \int_0^r \abs{u'( t- \tau)} \frac{d\tau}{\tau^{\alpha}} \) r^{\alpha}\frac{dr}{r} \\
	&\leq  \int_0^1 r^{\alpha-1} \int_0^1 \abs{u'(t-\tau)} \frac{d\tau}{\tau^{\alpha}} \, dr \\
	&= \frac{1}{\alpha} \int_0^1 \abs{u'(t-\tau)} \frac{d\tau}{\tau^{\alpha}}.
\end{aligned}
\end{equation}
Then, if we let $\eta(t) = t^{-\alpha}\chi_{(0,1)}(t)$, by Lemma \ref{LemmaLorente},
\begin{align*}
\abs{I_\alpha}
	&\leq \frac{1}{\abs{\Gamma(1-\alpha)}} (|u'|*\eta)(t) \\
	&\leq \frac{1}{\abs{\Gamma(1-\alpha)}} M^-u(t) \int_0^{\infty} \eta(\tau) \, d \tau=C_1M^-u(t)
\end{align*}
where
$$C_1= \frac{1}{\Gamma(2-\alpha)}.$$
Considering now the second integral in \eqref{eq:splitting}, we observe that
$$II_\alpha = \frac{1}{\Gamma(-\alpha)} \int_1^{\infty} u(t-\tau)  \frac{d \tau}{\tau^{1 + \alpha}}+  \frac{1}{\Gamma(1-\alpha)} \, u(t).$$
For the first term, we estimate using Lemma \ref{LemmaLorente} with $\eta(t)= \chi_{(0,1]}(t)+ t^{-1-\alpha}\chi_{(1,\infty)}(t)$,
$$\abs{ \frac{1}{\Gamma(-\alpha)}\int_1^{\infty} u(t-\tau)  \frac{d \tau}{\tau^{1 + \alpha}} } 
\leq \frac{1}{\abs{\Gamma(-\alpha)}} (\abs{u}*\eta)(t)\leq C_2M^-u(t)$$
where
$$C_2 = \frac{1+\alpha}{\Gamma(1-\alpha)}$$
which is bounded independently of $\alpha$.
Therefore,
$$\abs{II_\alpha} \leq  C_2  M^-u(t) +C_3\abs{u(t)}\leq (C_2 + C_3) M^-u(t)$$
where
$$C_3 = \frac{1}{\abs{\Gamma(1-\alpha)}}.$$
The result follows. 
\end{proof}

\section{Proof of Theorem \ref{thm:derivative1}}\label{Section:proofs derivatives}

\subsection{Proof of Theorem \ref{thm:derivative1}$(a)$}

The proof of part $(a)$ is organized as follows. 
We first show that the formula in the right hand side of \eqref{eq:formulaDleftalpha} is well-defined
as a function in $L^p(\R, \w)$. It is then shown that the distribution $\dl u$ is indeed given by such pointwise formula using the fact that $C^{\infty}_c(\R)$ is dense in $W^{1,p}(\R, \w)$.
The $L^p(\R, \w)$ estimate in \eqref{eq:LpestimateDleftalpha} follows immediately from these steps of the proof.
Next, we show that the limit in \eqref{eq:limitalpha} holds in $L^p(\R,\w)$ for $u \in C^{\infty}_c(\R)$
and then use a density argument to show the result for $u \in W^{1,p}(\R, \w)$. 
The a.e.~convergence of \eqref{eq:limitalpha} is proved by showing that the set of functions in $W^{1,p}(\R, \w)$
such that \eqref{eq:limitalpha} holds a.e.~is closed in $W^{1,p}(\R,\w)$. 
The a.e.~convergence of \eqref{eq:limitalpha2} follows similarly. Finally, the maximal estimate allows us to
prove that \eqref{eq:limitalpha2} holds in $L^p(\R,\w)$, $1<p<\infty$.

\medskip

\noindent\underline{\textbf{Step 1.}} The integral expression in \eqref{eq:formulaDleftalpha} defines a function in $L^p(\R, \w)$.

\smallskip

First let $1 < p < \infty$. By Theorem \ref{Prop:lp max} and the boundedness of $M^-$ in $L^p(\R,\w)$ for $\w\in A_p^-(\R)$,
it is immediate that
\begin{equation}\label{eq:bound}
\norm{\frac{1}{\Gamma(-\alpha)} \int_{0}^{\infty} \( u(t-\tau) - u(t)\) \frac{d \tau}{\tau^{1 + \alpha}}}_{L^p(\R,\w)}
 \leq C_\w \( \norm{u}_{L^p(\R,\w)} + \norm{u'}_{L^p(\R,\w)}\).
\end{equation}

For $p=1$, we consider the terms $I_\alpha$ and $II_\alpha$ as in \eqref{eq:splitting}. We use \eqref{eq:Taylor's Remainder} to observe that
\begin{align*}
\norm{I_\alpha}_{L^1(\R,\w)} 
	&\leq \frac{1}{\abs{\Gamma(1-\alpha)}} \int_{-\infty}^{\infty} \int_0^1 \abs{u'(t-\tau)} \frac{d\tau}{\tau^{\alpha}}\, \w(t) \, dt \\
&= \frac{1}{\abs{\Gamma(1-\alpha)}} \int_{-\infty}^{\infty} \int_{t-1}^t \frac{\abs{u'(\tau)}}{(t-\tau)^{\alpha}} \, d\tau \, w(t) \, dt\\
	&= \frac{1}{\abs{\Gamma(1-\alpha)}} \int_{-\infty}^{\infty} \abs{u'(\tau)} \int_\tau^{\tau+1} \frac{\w(t)}{(t-\tau)^{\alpha}} \, dt \, d\tau\\
	&= \frac{1}{\abs{\Gamma(1-\alpha)}} \int_{-\infty}^{\infty} \abs{u'(\tau)} \int_0^{1} \frac{\w(t+\tau)}{t^{\alpha}} \, dt \, d\tau.
\end{align*}
Since $\w\in A_1^-(\R)$, for a.e.~$\tau \in \R$ we can use Lemma \ref{LemmaLorente} with $\eta(\tau) = \abs{\tau}^{-\alpha}\chi_{(-1,0)}(\tau)$ to get
$$\int_0^{1} \frac{\w(t+\tau)}{t^{\alpha}} \, dt= \int_{-1}^{0} \frac{\w(\tau-t)}{\abs{t}^{\alpha}} \, dt
= (\w*\eta)(\tau)\leq \frac{1}{(1-\alpha)} M^+\w(\tau)\leq \frac{C}{(1-\alpha)} \w(\tau).$$
Therefore, 
$$\norm{I_\alpha}_{L^1(\R,\w)} 
	\leq C_\w C_1 \norm{u'}_{L^1(\R,\w)}$$
where $C_1$ is as in the proof of Theorem \ref{Prop:lp max}.
Moving to the second term in \eqref{eq:splitting}, we write
\begin{align*}
II_\alpha
 &= \frac{1}{\Gamma(-\alpha)} \int_1^{\infty} u(t-\tau)  \frac{d \tau}{\tau^{1 + \alpha}} +  \frac{1}{\Gamma(1-\alpha)} u(t) 
\end{align*}
and estimate
\begin{align*}
\norm{\int_1^{\infty} u(t-\tau)  \frac{d \tau}{\tau^{1 + \alpha}} }_{L^1(\R,\w)}
	&\leq  \int_{-\infty}^{\infty} \int_1^{\infty} \frac{\abs{u(t-\tau)}}{\tau^{1+\alpha}}  d\tau \,\w(t) \, dt\\
	&= \int_{-\infty}^{\infty} \int_{-\infty}^{t-1} \frac{\abs{u(\tau)}}{(t-\tau)^{1+\alpha}}  d\tau \,\w(t) \, dt\\
	&=  \int_{-\infty}^{\infty} \abs{u(\tau)}\int_{\tau+1}^{\infty} \frac{\w(t)}{(t-\tau)^{1+\alpha}} \, dt \, d\tau\\
	&=\int_{-\infty}^{\infty} \abs{u(\tau)}\int_{1}^{\infty} \frac{\w(t+\tau)}{t^{1+\alpha}} \, dt \, d\tau.
\end{align*}
By using again the $A_1^-(\R)$-condition and Lemma \ref{LemmaLorente} with 
$\eta(\tau) = \chi_{[-1,0)}(\tau)+\abs{\tau}^{-1-\alpha}\chi_{(-\infty,-1)}(\tau)$, for a.e.~$\tau \in \R$,
\begin{align*}
\int_{1}^{\infty} \frac{\w(t+\tau)}{t^{1+\alpha}} \, dt
	= \int_{-\infty}^{-1} \frac{\w(\tau-t)}{\abs{t}^{1+\alpha}} \, dt 
	\leq (\w*\eta)(\tau) 
	\leq \frac{1+\alpha}{\alpha} M^+\w(\tau)\leq C\frac{1+\alpha}{\alpha} \w(\tau).
\end{align*}
Therefore, by collecting terms,
\begin{equation}\label{eq:constantIIalpha}
\begin{aligned}
\norm{II_\alpha}_{L^1(\R,\w)}
	&\leq \frac{1}{\abs{\Gamma(-\alpha)}} \norm{ \int_1^{\infty} u(t-\tau)  \frac{d \tau}{\tau^{1 + \alpha}}}_{L^1(\R,\w)} +  \frac{1}{\abs{\Gamma(1-\alpha)}} \norm{u}_{L^1(\R,\w)}\\
	&\leq C_\w(C_2+C_3)\norm{u}_{L^1(\R,\w)}
\end{aligned}
\end{equation}
where $C_2,C_3>0$ are as in the proof of Theorem \ref{Prop:lp max}.
Thus,
\begin{equation}\label{eq:bound1}
\begin{aligned}
\norm{\frac{1}{\Gamma(-\alpha)} \int_{0}^{\infty} \( u(t-\tau) - u(t)\) \frac{d \tau}{\tau^{1 + \alpha}} }_{L^1(\R,\w)} \leq C_\w \( \norm{u}_{L^1(\R,\w)} + \norm{u'}_{L^1(\R,\w)}\).
\end{aligned}
\end{equation}

Hence, the integral in \eqref{eq:formulaDleftalpha} is in $L^p(\R,\w)$ for $1 \leq p < \infty$.

\medskip

\noindent\underline{\textbf{Step 2.}} The distribution $\dl u$ coincides with the integral formula in \eqref{eq:formulaDleftalpha}.
Therefore $\dl u$ is in $L^p(\R, \w)$ and, by \eqref{eq:bound} and \eqref{eq:bound1},
\eqref{eq:LpestimateDleftalpha} holds.

\smallskip

To show \eqref{eq:formulaDleftalpha}, let $u_k \in C^{\infty}_c(\R)$ such that $u_k \to u$ in $W^{1,p}(\R,\w)$ as $k\to\infty$.
We may write
\begin{align*}
\dl u_k(t) = \frac{1}{\Gamma(-\alpha)} \int_0^{\infty} \( u_k(t-\tau) - u_k(t)\) \frac{d\tau}{\tau^{1+\alpha}}.
\end{align*}
Using \eqref{eq:bound} and \eqref{eq:bound1}, we can show that the formulas converge in norm. Indeed,
\begin{align*}
&\norm{ \frac{1}{\Gamma(-\alpha)} \int_0^{\infty} \( u_k(t-\tau) - u_k(t)\) \frac{d\tau}{\tau^{1+\alpha}} -  \frac{1}{\Gamma(-\alpha)} \int_0^{\infty} \( u(t-\tau) - u(t)\) \frac{d\tau}{\tau^{1+\alpha}}}_{L^p(\R,\w)}\\
 &\hspace{.5in} \leq C \( \norm{u_k - u}_{L^p(\w)} + \norm{u_k'-u'}_{L^p(\R,\w)} \) \to 0\quad\hbox{as}~k\to\infty.
\end{align*}
If $\varphi \in C^{\infty}_c(\R)$ and $A$ is such that $\supp \varphi \subset (-\infty,A]$, then $\varphi \in \S_-$ and ${\dr \varphi \in \S_-^{\alpha}}$ with $\supp((D_\r)^\alpha\varphi)\subset(-\infty,A]$. Now, by Definition \ref{definition},
\begin{equation}\label{eq:distributionalcomputation}
\begin{aligned}
(\dl u)(\varphi)
	&= \int_{-\infty}^{\infty} u(t) \, \dr \varphi (t) \, dt \\
	&= \lim_{k \to \infty} \int_{-\infty}^{\infty} u_k(t) \, \dr \varphi(t) \, dt \\
	&= \lim_{k \to \infty} \int_{-\infty}^{\infty} \dl u_k(t) \, \varphi(t) \, dt \\
	&= \lim_{k \to \infty} \int_{-\infty}^{\infty} \( \frac{1}{\Gamma(-\alpha)} \int_0^{\infty} \( u_k(t-\tau) - u_k(t)\) \frac{d\tau}{\tau^{1+\alpha}}\) \varphi(t) \, dt \\
	&= \int_{-\infty}^{\infty}\( \frac{1}{\Gamma(-\alpha)} \int_0^{\infty} \( u(t-\tau) - u(t)\) \frac{d\tau}{\tau^{1+\alpha}} \)\varphi(t) \, dt.
\end{aligned}
\end{equation}
In the second identity above we used that, by Proposition \ref{prop:A},
\begin{align*}
\bigg|\int_{-\infty}^{\infty} u_k(t) \, \dr \varphi(t) \, dt - & \int_{-\infty}^{\infty} u(t) \, \dr \varphi (t) \, dt \bigg| \\
	&\leq \int_{-\infty}^{A} \abs{u_k(t)-u(t)} \, \abs{\dr \varphi (t)} \, dt \\
	&\leq C \int_{-\infty}^{A} \frac{\abs{u_k(t)-u(t)}}{1 + \abs{t}^{1+\alpha}} \, dt
	\leq C  \norm{u_k - u}_{L^p(\R,\w)} \to 0
\end{align*}
as $k \to \infty$ and in the last equality we observed that
\begin{align*}
&\abs{ \int_{-\infty}^{\infty}(D_\l)^\alpha u_k(t) \varphi(t) \, dt \ -  \int_{-\infty}^{\infty} \( \frac{1}{\Gamma(-\alpha)} \int_0^{\infty} \( u(t-\tau) - u(t)\) \frac{d\tau}{\tau^{1+\alpha}}\) \varphi(t) \, dt }\\
&\hspace{.5in} \leq C \int_{-\infty}^{A} \abs{ \int_0^{\infty} \( u_k(t-\tau) - u_k(t)\) \frac{d\tau}{\tau^{1+\alpha}}  - \int_0^{\infty} \( u(t-\tau) - u(t)\) \frac{d\tau}{\tau^{1+\alpha}} } \frac{1}{1+ \abs{t}^{1+\alpha}}\, dt \\
&\hspace{.5in} \leq C \norm{  \int_0^{\infty} \( u_k(\cdot-\tau) - u_k(\cdot)\) \frac{d\tau}{\tau^{1+\alpha}}  -  \int_0^{\infty} \( 
u(\cdot-\tau) - u(\cdot)\) \frac{d\tau}{\tau^{1+\alpha}} }_{L^p(\R,\w)} \to 0
\end{align*}
as $k\to\infty$. Therefore, since $\varphi$ was arbitrary in \eqref{eq:distributionalcomputation},
$$\dl u(t) = \frac{1}{\Gamma(-\alpha)} \int_0^{\infty} \( u(t-\tau) - u(t)\) \frac{d\tau}{\tau^{1+\alpha}} \quad\hbox{a.e. in}~\R.$$

\medskip

\noindent\underline{\textbf{Step 3.}} The limit as $\alpha \to 1^-$ in \eqref{eq:limitalpha} holds in $L^p(\R, \w)$ for $u \in C^{\infty}_c(\R)$.

\smallskip

Suppose that $u \in C^{\infty}_c(\R)$ and write $\dl u(t)=I_\alpha+II_\alpha$ as in \eqref{eq:splitting}.
For $1 < p < \infty$, we see from the proof of Theorem \ref{Prop:lp max} that
\begin{align*}
\norm{II_{\alpha}}_{L^p(\R,\w)}
	&\leq (C_2 +C_3)\norm{M^-{u}}_{L^p(\R,\w)}\\
	&\leq \(\frac{1+\alpha}{\Gamma(1-\alpha)} + \frac{1}{\Gamma(1-\alpha)} \)C_{\w}\norm{u}_{L^p(\R,\w)} \to 0
\end{align*}
as $\alpha\to1^-$.
For $p=1$, by \eqref{eq:constantIIalpha} in Step 1, we similarly obtain
$$\norm{II_{\alpha}}_{L^p(\w)}\leq C_w(C_2+C_3)\norm{u}_{L^1(\w)}\to0\quad\hbox{as}~\alpha\to1^-.$$
Next, observe that
\begin{align*}
I_{\alpha} - u'(t)
	&= \frac{1}{\Gamma(-\alpha)} \int_0^1 \( - \int_{0}^{\tau} u'(t-r) \, dr\)  \, \frac{d\tau}{\tau^{1+\alpha}} - u'(t)\\
	&= \frac{1}{\Gamma(-\alpha)} \int_0^1  \int_{0}^{\tau}(u'(t)- u'(t-r)) \, dr \, \frac{d\tau}{\tau^{1+\alpha}} +
	 \(\frac{\alpha}{\Gamma(2-\alpha)} - 1\) u'(t)\\
	&= \frac{1}{\Gamma(-\alpha)} \int_0^1 \int_{0}^{\tau} \int_{0}^r u''(t-\mu) \, d\mu \, dr \, \frac{d\tau}{\tau^{1+\alpha}} + \(\frac{\alpha}{\Gamma(2-\alpha)} - 1\) u'(t).
\end{align*}
Let $K$ be such that $\supp u''(\cdot -\mu) \subset [-K,K]$ for all $\mu \in [0,1]$. Then, for $1 \leq p < \infty$,
$$\norm{u''(\cdot -\mu)}_{L^p(\R,\w)}\leq\norm{u''}_{L^\infty(\R)}\w([-K, K])^{1/p}=c$$
where $c>0$ is independent of $\alpha$. Therefore,
\begin{align*}
\|I_{\alpha} -& u'\|_{L^p(\R, \w)} \\
	&\leq \frac{1}{\abs{\Gamma(-\alpha)}} \int_0^1 \int_{0}^{\tau} \int_{0}^r \norm{u''(t-\mu)}_{L^p(\R,\w)} \, d\mu \, dr  \, \frac{d\tau}{\tau^{1+\alpha}} + \abs{\frac{\alpha}{\Gamma(2-\alpha)} - 1} \norm{u'}_{L^p(\R,\w)}\\
	&= c \frac{\alpha (1-\alpha)}{\abs{\Gamma(3-\alpha)}} + \abs{\frac{\alpha}{\Gamma(2-\alpha)} - 1} \norm{u'}_{L^p(\R,\w)}\to 0\quad\hbox{as}~\alpha\to1^-.
\end{align*}
Hence, $\norm{\dl u - u'}_{L^p(\R,\w)} \leq \norm{II_{\alpha}}_{L^p(\R,\w)} + \norm{I_{\alpha} - u'}_{L^p(\R,\w)} \to 0$,
as $\alpha\to1^-$.

\medskip

\noindent\underline{\textbf{Step 4.}}  The limit as $\alpha \to 1^-$  in \eqref{eq:limitalpha} holds in $L^p(\R, \w)$ for $u \in W^{1,p}(\R, \w)$.

\smallskip

Let $u_k \in C^{\infty}_c(\R)$ such that $u_k \to u$ in $W^{1,p}(\R,\w)$ as $k\to\infty$. We just observe that,
by the $L^p$ estimate \eqref{eq:LpestimateDleftalpha} (that was proved in Step 2), for $1\leq p<\infty$,
\begin{align*}
&\norm{\dl u - u'}_{L^p(\R,\w)}\\
	&\quad \leq \norm{\dl(u - u_k)}_{L^p(\R,\w)} + \norm{\dl u_k - u_k'}_{L^p(\R,\w)} + \norm{u_k' - u'}_{L^p(\R,\w)}\\
	&\quad \leq C \( \norm{u-u_k}_{L^p(\R,\w)} + \norm{(u- u_k)'}_{L^p(\R,\w)} \) +  \norm{\dl u_k - u_k'}_{L^p(\R,\w)}.
\end{align*}
Then take $k$ large and choose $\alpha$ close to $1^-$ (see Step 3).

\medskip

\noindent\underline{\textbf{Step 5.}}  The limit as $\alpha \to 1^-$  in  \eqref{eq:limitalpha} holds almost everywhere for $u \in W^{1,p}(\R, \w)$.

\smallskip

It follows from Theorem \ref{Prop:lp max} and the properties of $M^-$ that the operator $T^*$ defined by
$$T^*u(t) = \sup_{0 < \alpha < 1} \dl u(t)\quad\hbox{for}~u\in W^{1,p}(\R,\w)$$
satisfies the estimates
$$\norm{T^*u}_{L^p(\R,\w)} \leq  C_{p,\w} \norm{u}_{W^{1,p}(\R,\w)}\quad\hbox{for any}~u\in W^{1,p}(\R,\w),~1 < p < \infty$$
and
$$\w\big(\{t\in\R:|T^*u(t)|>\lambda\}\big)\leq\frac{C_\w}{\lambda}\|u\|_{W^{1,1}(\R,\w)}\quad\hbox{for any}~u\in W^{1,1}(\R,\w).$$
In particular, $T^*$ is bounded from $W^{1,p}(\R,\w)$ into weak-$L^p(\R,\w)$, for
any $1\leq p<\infty$. This in turn implies that the set
$$E= \{ u \in W^{1,p}(\R,\w) : \lim_{\alpha \to 1^-} \dl u(t) = u'(t) \,\, \hbox{a.e.}\}$$
is closed in $W^{1,p}(\R,\w)$. Since $C^\infty_c(\R)\subset E$, by density, we get $E=W^{1,p}(\R,\w)$.

\medskip

\noindent
\underline{\bf Step 6.} The limit as $\alpha \to 0^+$ in \eqref{eq:limitalpha2} holds almost everywhere for $u \in W^{1,p}(\R, \w)$.

\smallskip

As in Step 5, one can check that the set
$$E' = \{ u \in W^{1,p}(\R,\w) : \lim_{\alpha \to 0^+} \dl u(t) = u(t) \,\, \hbox{a.e.}\}$$
is closed in $W^{1,p}(\R,\w)$. Since $C^\infty_c(\R)\subset E'$, by density, we get $E'=W^{1,p}(\R,\w)$.

\medskip

\noindent
\underline{\bf Step 7.} The limit as $\alpha \to 0^+$ in \eqref{eq:limitalpha2} holds in $L^p(\R, \w)$.

\smallskip

By Theorem \ref{Prop:lp max}, for any $0<\alpha<1$,
\begin{align*}
\abs{\dl u(t)- u(t)}^p \w(t) 
	&\leq \(C_n (M(u')(t) + M u(t)) + \abs{u(t)}\)^p \w(t) \\
	&\leq C_{n,p} \((M(u')(t))^p + (Mu(t))^p\) \w(t).
\end{align*}
Therefore, by Step 6 and the Dominated Convergence Theorem, \eqref{eq:limitalpha2} holds.

\smallskip

The proof of Theorem \ref{thm:derivative1}, part $(a)$, is completed.\qed

\subsection{Proof of Theorem \ref{thm:derivative1}$(b)$}

This is proved through a distributional argument.

Suppose that $\dl u \to v$ in $L^p(\R,\w)$ as $\alpha \to 1^{-}$. Let $\varphi \in C^{\infty}_c(\R)$.
Let $A\in\R$ be such that $\supp \varphi \subset (-\infty, A]$, so that $\varphi \in \S_{-}$ and $\dr \varphi \in \S_-^{\alpha}$.
By Proposition \ref{prop:A},
\begin{align*}
\abs{\int_{-\infty}^{\infty} v(t) \, \varphi(t) \, dt -  \int_{-\infty}^{\infty} \dl u(t) \, \varphi(t) \, dt}
	 &\leq \int_{-\infty}^{A} \abs{v(t)-\dl u(t)}\, \frac{C}{1+\abs{t}} \, dt\\
	 &\leq C_{\varphi,A,\w,p} \norm{v - \dl u}_{L^p(\R, \w)}\to 0
\end{align*}
as $\alpha\to1^-$. With this and the definition of $\dl u$ we can write
\begin{align*}
\int_{-\infty}^{\infty} v \, \varphi \, dt	
	 &= \lim_{\alpha \to 1^-} \int_{-\infty}^{\infty} \dl u \, \varphi \, dt \\
	 &= \lim_{\alpha \to 1^-} \int_{-\infty}^{\infty} u \, \dr \varphi \, dt.
\end{align*}
Next, notice that, by Proposition \ref{prop:A},
\begin{align*}
\abs{u(t)} |\dr\varphi + \varphi'| 
	&\leq |u(t)| \frac{C_\varphi}{1 + |t|^{1+\alpha}} \chi_{(-\infty, A]}(t) \\
	&\leq C_{\varphi} \frac{|u(t)|}{1 + |t|}  \chi_{(-\infty, A]}(t) \in L^1(\R).
\end{align*}
Therefore, by the Dominated Convergence Theorem, as $\dr\varphi(t)\to-\varphi'(t)$ as $\alpha\to0^+$,
\begin{align*}
\lim_{\alpha \to 1^-} \bigg|\int_{-\infty}^{\infty} u(t) &\, \dr \varphi(t) \, dt+ \int_{-\infty}^{\infty} u(t) \, \varphi'(t) \, dt\bigg|\\
	&\leq \int_{-\infty}^{\infty} \lim_{\alpha \to 1^-} \abs{u(t)} \abs{ \dr \varphi(t)  +\varphi'(t)} \, dt  = 0
\end{align*}
Whence,
\begin{align*}
\int_{-\infty}^{\infty} v \, \varphi \, dt	
	 &= \lim_{\alpha \to 1^-} \int_{-\infty}^{\infty} u \, \dr \varphi \, dt \\
	 &= -\int_{-\infty}^{\infty} u \,  \varphi' \, dt= \int_{-\infty}^{\infty}  u' \, \varphi \, dt.
\end{align*}
Therefore $v = u'$ a.e.~in $\R$. Since $u'=v \in L^p(\R,\w)$, we get $u \in W^{1,p}(\R,\w)$, and by Theorem \ref{thm:derivative1}$(a)$, the conclusion follows. \qed

\subsection{Proof of Theorem \ref{thm:derivative1}$(c)$}

Using the exact same arguments as in part $(b)$, we find that
\begin{align*}
\int_{-\infty}^{\infty} v \, \varphi \, dt &= \lim_{\alpha\to0^+}\int_{-\infty}^\infty \dl u \, \varphi\,dt \\
	 &= \lim_{\alpha \to0^+} \int_{-\infty}^{\infty} u \, \dr \varphi \, dt = \int_{-\infty}^{\infty} u \,  \varphi\, dt.
\end{align*}
Therefore $v = u$ a.e. in $\R$ and the conclusion follows. \qed

\section{Fractional Laplacians and Muckenhoupt weights}\label{Section:prelim Laplacian}

For $u \in \S(\R^n)$, the Fourier transform identity
$$\widehat{(-\Delta) u}(\xi) = |\xi|^2 \widehat{u}(\xi)$$
is used to define the fractional Laplacian as
$$\widehat{(-\Delta)^s u}(\xi) = |\xi|^{2s} \widehat{u}(\xi)\quad\hbox{for}~0<s<1.$$
Using the heat diffusion semigroup $\{e^{t\Delta} \}_{t \geq 0}$ generated by $-\Delta$, it is shown in
\cite{Stinga,Stinga-Torrea} that the fractional Laplacian can be expressed using the semigroup formula \eqref{eq:laplacian semigroup}
and that this is equivalent to the pointwise formula \eqref{eq:laplacian pw}. 
Here, $e^{t\Delta}$ is the operator
$$\widehat{e^{t\Delta} u}(\xi) = e^{-t|\xi|^2} \widehat{u}(\xi).$$
It is well known that
$e^{t\Delta} u(x) =  (W_t * u)(x)$ where $W_t(x)$ is the Gauss--Weierstrass kernel
$$ W_t(x) = \frac{1}{(4 \pi t)^{n/2}} \, e^{-\abs{x}^2/(4t)}\qquad\hbox{for}~x\in\R^n,~t>0.$$

\subsection{Distributional setting}

The distributional setting for the fractional Laplacian was developed by Silvestre in \cite{Silvestre}.
Consider the function class
$$\S_{s} = \left\{ \varphi \in C^{\infty} (\R^n) :|D^{\gamma} \varphi(x)| \leq \frac{C}{1 + \abs{x}^{n+2s}},~\text{for all }
\gamma \in \N_0^n,~x\in\R^n,~\hbox{for some}~C>0\right\}.$$
We endow $\S_s$ with the topology induced by the family of seminorms
$$\rho_{\gamma}^{s} (\varphi) = \sup_{x \in \R^n} (1 + \abs{x}^{n+ 2 s}) \abs{D^{\gamma} \varphi(x)}, \quad \text{ for } \gamma \in \N_0^n.$$
Let $(\S_{s})'$ be the dual space of $\S_{s}$. Notice that $\S \subset \S_{s}$, so that $(\S_{s})' \subset \S'$. 
For $u \in (\S_{s})'$, $(-\Delta)^{s} u$ is defined as a distribution on $\S'$ by
$$((-\Delta)^{s} u)(\varphi) = u \( (-\Delta)^{s} \varphi \) \quad\hbox{for any}~\varphi\in \S.$$
One can check that $L_{s} \subset (\S_{s})'$, where
$$L_{s} = \left\{ u\in L^1_{\mathrm{loc}}(\R^n) : \int_{\R^n} \frac{\abs{u(x)}}{1 + \abs{x}^{n+2s}} \, dx < \infty \right\}.$$

\begin{prop}[Silvestre {\cite{Silvestre}}]
Let $\Omega$ be an open set in $\R^n$ and $u \in L_s$. If $u \in C^{2s + \varepsilon}(\Omega)$ (or $C^{1, 2s + \varepsilon -1}(\Omega)$
if $s \geq 1/2$) for some $\varepsilon >0$, then $(-\Delta)^su\in C(\Omega)$ and
$$(-\Delta)^{s} u(x) = c_{n,s} P.V. \int_{\R^n} \frac{u(x) - u(y)}{\abs{x-y}^{n + 2s}} \,dy \quad \text{ for every } x \in \Omega.$$
Here (see \cite{Stinga,Stinga-Torrea})
\begin{equation}\label{eq:cns}
c_{n,s} = \frac{4^s \Gamma(n/2+s)}{\abs{\Gamma(-s)} \pi^{n/2}}\sim s(1-s)\quad\hbox{as}~s\to0,1.
\end{equation}
\end{prop}

\subsection{Muckenhoupt weights}

A function $\nu\in L^1_{\mathrm{loc}}(\R^n)$, $\nu>0$ a.e., is called an $A_p(\R^n)$ Muckenhoupt weight, $1<p<\infty$,
if it satisfies the following condition: there exists $C>0$ such that
\begin{equation}\label{eq:muckenhoupt}
\( \frac{1}{\abs{B}} \int_B \nu\,dx \)^{1/p} \( \frac{1}{\abs{B}} \int_B  \nu^{1-p'}\,dx \)^{1/p'} \leq C
\end{equation}
for any ball $B \subset \R^n$. If $\nu$ satisfies (\ref{eq:muckenhoupt}), we write $\nu \in A_p(\R^n)$.
Observe that $\nu \in A_p(\R^n)$ if and only if $\nu^{1-p'} \in A_{p'}(\R^n)$.
The Hardy--Littlewood maximal function is defined by
$$Mu(x) = \sup_{B \ni x} \frac{1}{\abs{B}} \int_B \abs{u(y)} \, dy$$
where the supremum is taken over all balls $B \subset \R^n$ containing $x$. 
For $1 < p < \infty$, the operator $M$ is bounded on $L^p(\R^n,\nu)$ if and only if $\nu \in A_p(\R^n)$. 
When $p=1$, $M$ is bounded from $L^1(\R^n, \nu)$ into weak-$L^1(\R^n, \nu)$ if and only if $\nu \in A_1(\R^n)$, namely, there exists $C>0$ such that
$$M \nu(x) \leq C \nu(x) \quad\hbox{for a.e.}~x \in \R^n.$$
For a measurable set $E\subset\R^n$ and a weight $\nu$, we denote
$$\nu(E)=\int_E\nu\,dx.$$
See \cite{Duo} for more details about Muckenhoupt weights.

\begin{lem}[{See \cite[Proposition~2.7]{Duo}}]\label{LemmaDuo}
Let $\eta=\eta(x)$ be a function that is positive, radial, decreasing (as a function on $(0,\infty)$) and integrable.
Then for any measurable function $u:\R^n\to\R$ and for almost every $x\in\R^n$, we have
$$|u\ast\eta(x)|\leq \|\eta\|_{L^1(\R^n)}Mu(x).$$
\end{lem}

\begin{lem}[See {\cite[Corollary 7.6]{Duo}}]\label{prop:reverseHolder}
If $\nu \in A_p(\R^n)$, $1 \leq p < \infty$, then there exists $\delta >0$ such that given a ball $B$ and
a measurable subset $S$ of $B$, 
$$ \frac{\nu(S)}{\nu(B)} \leq C \( \frac{\abs{S}}{\abs{B}}\)^\delta.$$
\end{lem}

Our next result shows that for any function $u \in L^p(\R^n, \nu)$, $\nu \in A_p(\R^n)$, $1\leq p<\infty$, the object $(-\Delta)^s u $ is well defined as a distribution in $\S'$.

\begin{prop}\label{prop:LsLp}
If  $u\in L^p(\R^n, \nu)$, $\nu \in A_p(\R^n)$, $1\leq p<\infty$, then $u\in L_s$, $s\geq0$,
and there is a constant $C=C_{n,p,\nu}>0$ such that
$$\norm{u}_{L_s} \leq C \norm{u}_{L^p(\R^n, \nu)}.$$
In particular, $L^p(\R^n,\nu) \subset L^1_{\operatorname{loc}}(\R^n)$.
\end{prop}

\begin{proof}
Suppose first that $1 < p < \infty$. By H\"older's inequality,
$$\norm{u}_{L_s}
	\leq \norm{u}_{L^p(\R^n, \nu)} \(C_n \int_{\R^n} \frac{\nu(x)^{1-p'}}{(1 + \abs{x}^{n})^{p'}} \, dx\)^{\frac{1}{p'}}.$$
Let $\tilde{\nu}(x) = \nu(x)^{1-p'} \in A_{p'}(\R^n)$. 
It is enough to show
$$\int_{\R^n} \frac{\tilde{\nu}(x)}{(1+ \abs{x})^{np'}}\, dx < \infty.$$
Let $f(x) = \chi_{B_1}(x)$.  If $\abs{x} \leq 1$, then $Mf(x) = 1$. If $\abs{x} \geq 1$, then $B_1\subset B(x, 2\abs{x}) $ and
$$Mf(x) \geq \frac{\abs{B(0,1)}}{\abs{B(x, 2\abs{x})}} = \frac{1}{(2\abs{x})^n} \geq \frac{C_n}{(1+\abs{x})^n}.$$
Since $M$ is bounded on $L^{p'}(\R^n,\tilde{\nu})$, for $\tilde{\nu} \in A_{p'}(\R^n)$,
\begin{align*}
\int_{\R^n} \frac{\tilde{\nu}(x)}{(1+ \abs{x})^{np'}} \, dx &\leq C \int_{\R^n} (Mf(x))^{p'}\tilde{\nu}(x) \, dx \\
&\leq C \int_{\R^n} (f(x))^{p'} \tilde{\nu}(x) \, dx
= C \int_{B_1} \tilde{\nu}(x) \, dx= C\nu^{1-p'}(B_1).
\end{align*}
Therefore, $\norm{u}_{L_s} \leq C \norm{u}_{L^p(\R^n,\nu)}\nu^{1-p'}(B_1)$.

Now let $p = 1$. Observe that
$$\int_{\abs{x}<1} \frac{\abs{u(x)}}{1 + \abs{x}^{n+2s}} \, dx
	\leq \norm{u}_{L^1(\R^n,\nu)}\sup_{x \in B_1}  \nu(x)^{-1}\leq C_{n,\nu}\norm{u}_{L^1(\R^n,\nu)}$$
where in the last inequality we used that, since $\nu\in A_1(\R^n)$,
$$\sup_{B_1} \nu^{-1}  = \( \inf_{B_1} \nu\)^{-1} \leq C \( \frac{\nu(B_1)}{\abs{B_1}}\)^{-1}.$$
On the other hand, let $B_j = B_{2^j}(0)$, $j\geq0$. By using the $A_1(\R^n)$-condition and Lemma \ref{prop:reverseHolder}
with $S=B_1$ and $B=B_j$,
\begin{align*}
\int_{\abs{x} > 1} \frac{\abs{u(x)}}{1 + \abs{x}^{n+2s}} \, dx
	&\leq \sum_{j=0}^{\infty} \int_{B_{j+1} \setminus B_j} \frac{\abs{u(x)}}{ \abs{x}^{n}} \, dx\\
	&\leq  c_n\sum_{j=1}^{\infty} \frac{1}{( 2^j)^{n}} \int_{B_j} \abs{u(x)}\, dx\\
	&\leq c_{n}  \norm{u}_{L^1(\R^n, \nu)} \sum_{j=1}^{\infty} \frac{1}{( 2^j)^{n}}\sup_{x\in B_j} \nu^{-1}(x) \\
	&\leq C\norm{u}_{L^1(\R^n, \nu)} \sum_{j=1}^{\infty} \frac{1}{( 2^j)^{n}}\frac{\abs{B_j}^\delta}{\nu(B_j)}|B_j|^{1-\delta} \\
	&\leq C\norm{u}_{L^1(\R^n, \nu)}\frac{|B_1|^\delta}{\nu(B_1)} \sum_{j=1}^{\infty} \frac{(2^{jn})^{1-\delta}}{( 2^j)^{n}}
	\leq C_{n, \nu} \norm{u}_{L^1(\R^n, \nu)}.
\end{align*}
The result for $p=1$ follows by combining the previous estimates.
\end{proof}

\subsection{The heat semigroup on weighted spaces}

Recall the definition of the classical heat semigroup $\{e^{t\Delta}\}_{t\geq0}$ on $\R^n$:
\begin{equation}\label{eq:etDeltau}
e^{t\Delta}u(x)=\int_{\R^n}W_t(x-y)u(y)\,dy=\frac{1}{(4\pi t)^{n/2}}\int_{\R^n}e^{-|x-y|^2/(4t)}u(y)\,dy
\end{equation}
for $x\in\R^n$, $t>0$. We believe that the following result belongs to the folklore, but we provide a proof
for the sake of completeness.

\begin{thm} \label{Prop:semigroup}
Let $\nu \in A_p(\R^n)$ and $u \in L^p(\R^n, \nu)$, $1 \leq p < \infty$. The following hold.
\begin{enumerate}[$(1)$]
	\item The integral defining $e^{t\Delta} u(x)$ in \eqref{eq:etDeltau} is absolutely convergent for $x\in\R^n$, $t>0$,
	and
	$$ \sup_{t>0} \abs{e^{t\Delta}u(x)} \leq Mu(x)$$
	for almost every $ x \in \R^n $.

	\item $e^{t\Delta}u(x)\in C^{\infty}((0, \infty) \times \R^n)$ and $\partial_t(e^{t\Delta}u)=\Delta(e^{t\Delta}u)$ in $\R^n\times(0,\infty)$.

	\item  $\displaystyle{\norm{e^{t\Delta}u}_{L^p(\R^n,\nu)} \leq C_{n,p,\nu} \norm{u}_{L^p(\R^n,\nu)}}$, where $C_{n,p,\nu}>0$.
		
	\item $\displaystyle{\lim_{t \to 0^+} e^{t\Delta} u(x) = u(x)}$ for almost every $ x \in \R^n $.
	
	\item $\displaystyle{ \lim_{t \to 0^+} \norm{e^{t \Delta} u - u}_{L^p(\R^n, \nu)}= 0}$. 
	
	\item If $u \in W^{2,p}(\R^n,\nu)$, then $e^{t \Delta} \Delta u = \Delta e^{t \Delta} u$.
	
	\item $\displaystyle{\lim_{\varepsilon \to 0} \norm{\int_{\abs{x-y}<\varepsilon} W_t(x-y) u(y) \, dy}_{L^p(\R^n, \nu)} = 0}$.
\end{enumerate}
\end{thm}

\begin{proof}
Let $u\in L^p(\R^n,\nu)$, $\nu\in A_p(\R^n)$, for $1\leq p<\infty$.

For $(1)$, we apply Lemma \ref{LemmaDuo} with $\eta(x)=W_t(x)$ and notice that $\|W_t\|_{L^1(\R^n)}=1$,
for each fixed $t>0$.

To prove $(2)$, we recall that $W_t(x)\in C^\infty(\R^n\times(0,\infty))$, $\partial_tW_t=\Delta W_t$ in $\R^n\times(0,\infty)$
and that there exists $c>0$ such that $|\partial_tW_t(x)|\leq\frac{c}{t}W_{ct}(x)$ for each $t>0$ and $x\in\R^n$.
Thus, we can differentiate inside of the integral in \eqref{eq:etDeltau} to find that $e^{t\Delta}u(x)\in C^\infty(\R^n\times(0,\infty))$
and solves the heat equation.

If $1<p<\infty$, then part $(1)$ and the boundedness of the maximal function $M$ show that
$\norm{e^{t\Delta}u}_{L^p(\R^n, \nu)} \leq C \norm{u}_{L^p(\R^n, \nu)}$.
If $p=1$, as in part $(1)$ and by using the $A_1(\R^n)$-condition,
\begin{align*}
\norm{e^{t\Delta}u}_{L^1(\R^n, \nu)}
	&\leq \int_{\R^n}  \abs{u(y)} \( \int_{\R^n}  W_t(x-y) \nu(x) \, dx \) dy
	\leq \int_{\R^n} \abs{u(y)} M \nu(y) \, dy\\
	&\leq C \int_{\R^n} \abs{u(x)} \nu(y) \, dy= C \norm{u}_{L^1(\R^n, \nu)}.
\end{align*}
Whence, $(3)$ holds.

To verify the almost everywhere limit in $(4)$, we only need to observe that $\lim_{t\to0^+}e^{t\Delta}\varphi(x)=\varphi(x)$
for every $x\in\R^n$ whenever $\varphi\in C^\infty_c(\R^n)$, that $C^\infty_c(\R^n)$ is dense in $L^p(\R^n,\nu)$
and that, by part $(1)$, the maximal operator
$$
T^* u(x) = \sup_{t>0} \abs{e^{t \Delta}u(x)}
$$
is bounded from $L^p(\R^n,\nu)$ into weak-$L^p(\R^n,\nu)$ (see, for instance, \cite[Theorem~2.2]{Duo}).

For $(5)$, notice that if $\varphi\in C^\infty_c(\R^n)$, then, as in part $(1)$,
\begin{align*}
\abs{e^{t \Delta}\varphi(x) - \varphi(x)}
	&= \abs{ \int_0^t \partial_s e^{s \Delta} \varphi(x) \, ds}\\
	&\leq \int_0^t \abs{e^{s \Delta} \Delta \varphi(x)} \, ds
	\leq C M({\Delta \varphi})(x) \, t.
\end{align*}
For $1 < p < \infty$,
$$\norm{e^{t\Delta} \varphi - \varphi}_{L^p(\R^n, \nu)} 
	\leq C \norm{\Delta \varphi}_{L^p(\R^n, \nu)} t
	\to 0$$
as $t \to 0^+$. 
If $p=1$, then by part $(3)$, 
\begin{align*}
\norm{e^{t\Delta} \varphi - \varphi}_{L^1(\R^n, \nu)} 
	&= \int_{\R^n}  \abs{e^{t \Delta}\varphi(x) - \varphi(x)} \nu(x) \, dx \\
	&\leq  \int_{\R^n} \int_0^t \abs{e^{s \Delta} \Delta \varphi(x)} \nu(x) \, ds \, dx \\
	&= \int_0^t \norm{e^{s\Delta} \Delta \varphi}_{L^1(\R^n, \nu)} \, ds\\
	&\leq  \int_0^t C \norm{\Delta \varphi}_{L^1(\R^n, \nu)} \, ds 
	=  C t\norm{\Delta \varphi}_{L^1(\R^n, \nu)} \to 0
\end{align*}
as $t \to 0^+$. We then use the density of $C^\infty_c(\R^n)$ in $L^p(\R^n,\nu)$.

For $(6)$, let $\varphi \in C^{\infty}_c(\R^n)$ and observe that
\begin{align*}
\int_{\R^n}  \Delta e^{t \Delta} u(x) \varphi(x) \, dx
	&= \int_{\R^n}  e^{t \Delta} u(x) \Delta \varphi(x) \, dx\\
	&= \int_{\R^n} \int_{\R^n} W_t(x-y) u(y) \Delta \varphi(x) \, dy \, dx\\
	&= \int_{\R^n} W_t(z)\bigg[\int_{\R^n}  u(y) \Delta_y \varphi(x+y) \, dy \bigg]\, dx \\
	&= \int_{\R^n} W_t(z)\bigg[\int_{\R^n}  \Delta u(y) \varphi(x+y) \, dy \bigg]\, dx \\
	&= \int_{\R^n} e^{t \Delta} \Delta u(x) \, \varphi(x) \, dx.
\end{align*}
Then $\Delta e^{t\Delta}u(x)=e^{t\Delta}\Delta u(x)$, for almost every $x \in \R^n$.

Let us finally prove $(7)$. Observe that, by part $(1)$,
$$\int_{\abs{x-y}< \varepsilon} W_t(x-y) \abs{u(y)}  \, dy  \leq Mu(x).$$
 For $1<p<\infty$,  $Mu \in L^p(\R^n, \nu)$ so, by the Dominated Convergence Theorem,
\begin{align*}
\lim_{\varepsilon \to 0} \bigg\|\int_{\abs{x-y}< \varepsilon}& W_t(x-y) \abs{u(y)}  \, dy\bigg\|_{L^p(\R^n, \nu)}^p\\
 &= \int_{\R^n} \lim_{\varepsilon \to 0}\( \int_{\abs{x-y}< \varepsilon} W_t(x-y) \abs{u(y)}  \, dy\)^p \nu(x) \, dx = 0.
 \end{align*}
For $p=1$,
\begin{align*}
\bigg\|\int_{\abs{x-y}< \varepsilon} W_t(x-y) {u(y)}  \, dy \bigg\|_{L^1(\R^n, \nu)}
	&\leq  \int_{\R^n} \bigg[\abs{u(y)} \int_{\abs{x-y}< \varepsilon} W_t(x-y)  \nu(x) \, dx \bigg]\, dy
\end{align*}
and, by part $(1)$,
\begin{align*}
\abs{u(y)} \int_{\abs{x-y}< \varepsilon} W_t(x-y)  \nu(x) \, dx 
	&\leq  \abs{u(y)}M\nu(y) 
\leq C \abs{u(y)} \nu(y) \in L^1(\R^n)
\end{align*}
for a.e.~$y \in \R^n$. Therefore, $(7)$ holds for $p=1$ by the Dominated Convergence Theorem.
\end{proof}

\subsection{The maximal estimate \eqref{eq:Ms}}

\begin{thm}\label{prop:Mbound}
There exists a constant $C_{n}>0$ such that for any
$u \in W^{2,p}(\R^n,\nu)$, $\nu\in A_p(\R^n)$, $1\leq p<\infty$, we have
$$\sup_{0 < s < 1}\sup_{\varepsilon>0} \abs{c_{n,s} \int_{\abs{y}>\varepsilon} \frac{u(x-y) - u(x)}{\abs{y}^{n+2s}} \, dy} \leq C_n \( M (D^2 u)(x) + Mu(x)\)$$
for almost every $x\in\R^n$.
\end{thm}

\begin{proof}
Define the operator $T_{s,\varepsilon}$ on $W^{2,p}(\R^n, \nu)$ by
$$T_{s, \varepsilon}u(x) =  c_{n,s} \int_{\abs{y}>\varepsilon} \frac{u(x-y) - u(x)}{\abs{y}^{n+2s}} \, dy.$$
We will show that there is a constant $C= C_n>0$ such that
$$\abs{T_{s,\varepsilon} u(x)} \leq C\( M ({D^2 u})(x) + Mu(x)\) \quad\hbox{for a.e.}~ x \in \R^n$$
from which the statement follows. We write
$$T_{s,\varepsilon}u(x)
	= c_{n,s} \int_{\varepsilon < \abs{y} <1}  \frac{u(x-y) - u(x)}{\abs{y}^{n+2s}} \, dy + c_{n,s} \int_{\abs{y} >1}  \frac{u(x-y) - u(x)}{\abs{y}^{n+2s}} \, dy
	= I + II.$$
Let us first estimate the second term. Take $\eta(x)=\chi_{\{\abs{x}\leq1\}}(x)+\abs{x}^{-n-2s} \chi_{\{\abs{x}>1\}}(x)$ in Lemma \ref{LemmaDuo}
and use \eqref{eq:cns} to get
\begin{align*}
\abs{II}&\leq c_{n,s} \int_{\abs{y}>1} \frac{\abs{u(x-y)}}{\abs{y}^{n+2s}} \, dy + c_{n,s} \abs{u(x)} \int_{\abs{y}>1} \frac{1}{\abs{y}^{n + 2s}} \,  dy \\
	&\leq C_ns(1-s)\bigg((\abs{u}*\eta)(x)+\frac{|u(x)|}{s}\bigg) \\
	&\leq C_ns(1-s)\bigg(\( \frac{1+2s}{2s} \) Mu(x)+\frac{|u(x)|}{s}\bigg) \\
	&\leq C_n Mu(x).
\end{align*}
Consider now the first term, that we rewrite as
$$I = c_{n,s} \int_{\varepsilon < \abs{y} <1}  \frac{u(x-y) - u(x) + \nabla u(x) \cdot y}{\abs{y}^{n+2s}} \, dy.$$
Since $u\in W^{2,p}(\R^n,\nu)$ and \eqref{eq:cns} holds, for a.e.~$x\in\R^n$ we can estimate
\begin{align*}
\abs{I}
	&\leq c_{n,s}  \int_{\varepsilon < \abs{y} <1} \frac{\abs{u(x-y) - u(x) + \nabla u(x) \cdot y}}{\abs{y}^{n+2s}} \, dy\\
	&\leq c_{n,s} \int_{\varepsilon < \abs{y} <1} \frac{\abs{y}^2}{\abs{y}^{n+2s}} \int_{0}^1 (1-t)\, \abs{D^2u(x-ty)} \, dt \, dy \\
	&= c_{n,s} \int_0^1 (1-t) \int_{\varepsilon < \abs{y} < 1} \frac{\abs{D^2u(x-ty)}}{\abs{y}^{n-2(1-s)}} \, dy \, dt \\
	&\leq c_{n,s} \int_0^1 (1-t)\,t^{-2(1-s)}  \int_{\abs{y} < t}\frac{\abs{ D^2u(x-y)}}{\abs{y}^{n - 2(1-s)}} \, dy \, dt \\
	&\leq c_{n,s} \int_0^1 (1-t)\,t^{-2(1-s)} \sum_{k=0}^{\infty}  \int_{2^{-(k+1)}t  < \abs{y} < 2^{-k}t}\frac{\abs{ D^2u(x-y)}}{\abs{y}^{n - 2(1-s)}} \, dy \, dt\\
	&\leq c_{n,s} \int_0^1 (1-t)\,t^{-2(1-s)} \sum_{k=0}^{\infty} \frac{1}{(2^{-(k+1)}t)^{n-2(1-s)}} \int_{\abs{y} < 2^{-k}t} \abs{ D^2u(x-y)} \, dy \, dt\\
	&\leq C_ns(1-s)2^{n-2(1-s)}M({D^2 u})(x) \int_0^1 (1-t)\bigg[ \sum_{k=0}^{\infty}  \frac{1}{\(2^{2(1-s)}\)^k}\bigg] \, dt\\
	&\leq C_n\frac{s(1-s)}{4^{1-s}-1}M({D^2u})(x)\leq C_nM(D^2u)(x)
\end{align*}
where in the last line we applied the estimate $4^{1-s}-1\geq c(1-s)$, for any $0<s<1$.
Therefore, $\abs{T_{s,\varepsilon}u(x)} \leq \abs{I} + \abs{II} \leq C_n\(M({D^2u})(x) +  Mu(x)\)$ for a.e $x\in\R^n$.
\end{proof}

\section{Proof of Theorem \ref{thm:laplacian1}}\label{Section:proofs Laplacians}

\subsection{Proof of Theorem \ref{thm:laplacian1} $(a)$}

The steps in the proof of part $(a)$ are similar to the steps in the proof of Theorem \ref{thm:derivative1} $(a)$. 
 
\medskip

\noindent
\underline{\bf Step 1.} The semigroup formula in \eqref{eq:laplacian semigroup} defines a function in $L^p(\R^n, \nu)$.

\smallskip

Let us begin by writing 
\begin{equation}\label{eq:delta Is IIs}
\begin{aligned}
\frac{1}{\Gamma(-s)} \int_0^{\infty} &| e^{t \Delta} u(x) - u(x)| \frac{dt}{t^{1+s}} \\
&= \frac{1}{\Gamma(-s)} \int_0^1 | e^{t \Delta} u(x) - u(x)| \frac{dt}{t^{1+s}} + \frac{1}{\Gamma(-s)} \int_1^{\infty} | e^{t \Delta} u(x) - u(x)| \frac{dt}{t^{1+s}} \\
&= I + II.
\end{aligned}
\end{equation}
To study $I$, recall Theorem \ref{Prop:semigroup} and observe for $t \in [0,1]$ that
$$\norm{e^{t\Delta} u - u}_{L^p(\R^n, \nu)} 
	\leq \int_0^t \norm{e^{r \Delta}( \Delta u)}_{L^p(\R^n, \nu)} \, dr 
	\leq C \norm{\Delta u}_{L^p(\R^n, \nu)} t.$$
Therefore
\begin{equation}\label{eq:tbound}
\begin{aligned}
\norm{I}_{L^p(\R^n, \nu)}
	&\leq \frac{1}{\abs{\Gamma(-s)}} \int_0^1 \norm{e^{t\Delta}u - u}_{L^p(\R^n, \nu)} \frac{dt}{t^{1+s}}\\
	&=\frac{C}{\abs{\Gamma(-s)}} \,\norm{\Delta u}_{L^p(\R^n, \nu)} \int_0^1 t^{-s} \, dt
	= C\frac{s}{\abs{\Gamma(2-s)} } \, \norm{\Delta u}_{L^p(\R^n, \nu)}.
\end{aligned}
\end{equation}
For $II$, in view of Theorem \ref{Prop:semigroup},
\begin{equation}\label{eq:tbound2}
\begin{aligned}
\norm{II}_{L^p(\R^n, \nu)}
	&\leq \frac{1}{\abs{\Gamma(-s)}} \int_1^{\infty} \( \norm{e^{t \Delta} u}_{L^p(\R^n, \nu)} + \norm{u}_{L^p(\R^n, \nu)}\) \frac{dt}{t^{1+s}} \\
	&\leq \frac{1}{\abs{\Gamma(-s)}} \(C \norm{u}_{L^p(\R^n, \nu)} + \norm{u}_{L^p(\R^n, \nu)}\)\int_1^{\infty} \, \frac{dt}{t^{1+s}} \\
	&= \frac{C(1-s)}{\abs{\Gamma(2-s)}} \norm{u}_{L^p(\R^n, \nu)}.
\end{aligned}
\end{equation}
Therefore
\begin{equation}\label{eq:bound2}
\norm{\frac{1}{\Gamma(-s)} \int_0^{\infty} \( e^{t \Delta} u(x) - u(x) \) \frac{dt}{t^{1+s}}}_{L^p(\R^n, \nu)} \leq C \( \norm{u}_{L^p(\R^n, \nu)} + \norm{\Delta u}_{L^p(\R^n, \nu)}\) < \infty.
\end{equation}

\medskip

\noindent
\underline{\bf Step 2.} The distribution $(-\Delta)^su$ coincides with the semigroup formula in  \eqref{eq:laplacian semigroup}
for a.e. $x\in\R^n$. Therefore, $(-\Delta)^su$ is in $L^p(\R, \nu)$ and, by \eqref{eq:bound2}, we see that \eqref{eq:lpestimatelaplacian} holds. 

\smallskip

Since $C_c^{\infty}(\R^n)$ is dense in $W^{2,p}(\R^n, \nu)$ (see \cite{Turesson}), there exists a sequence $u_k \in C_c^{\infty}(\R^n)$ such that $u_k \to u$ in $W^{2,p}(\R^n, \nu)$. We consider the terms $I$ and $II$
as in \eqref{eq:delta Is IIs}
and, similarly,
\begin{align*}
(-\Delta)^s u_k(x) 
	&= \frac{1}{\Gamma(-s)} \int_0^{1} \( e^{t \Delta} u_k(x) - u_k(x)\) \frac{dt}{t^{1+s}} + \frac{1}{\Gamma(-s)} \int_1^{\infty} \( e^{t \Delta} u_k(x) - u_k(x)\) \frac{dt}{t^{1+s}} \\
	&= I_k + II_k.
\end{align*}
By \eqref{eq:tbound},
$$\norm{I_k - I}_{L^p(\R^n, \nu)}
	\leq C\frac{s}{\abs{\Gamma(2-s)}} \norm{\Delta (u_k - u)}_{L^p(\R^n, \nu)} \to 0 \quad\hbox{as}~k \to \infty.$$
Similarly, by \eqref{eq:tbound2},
$$\norm{II_k - II}_{L^p(\R^n, \nu)}
	= \frac{C(1-s)}{\Gamma(2-s)} \norm{u_k - u}_{L^p(\R^n, \nu)}\to 0  \quad\hbox{as}~k \to \infty.$$
Therefore,
\begin{equation}\label{eq:uktou}
(-\Delta)^s u_k(x)  \to  \frac{1}{\Gamma(-s)} \int_0^{\infty} \(e^{t \Delta}u(x) - u(x)\) \, \frac{dt}{t^{1+s}}
\end{equation}
in $L^p(\R^n, \nu)$ as $k \to \infty$.

Next, let $\varphi \in C^{\infty}_c(\R^n)$ and note that $(-\Delta)^s\varphi \in \S_s$. 
By Proposition \ref{prop:LsLp},
\begin{align*}
\bigg|\int_{\R^n} u_k(x) (-\Delta)^s \varphi(x) \, &dx -  \int_{\R^n} u(x) (-\Delta)^s \varphi(x) \, dx \bigg|\\
	&\leq C  \int_{\R^n} \frac{\abs{u_k(x) - u(x)}}{ 1+\abs{x}^{n+2s}} \, dx\\
	&\leq C \norm{u_k - u}_{L^p(\R^n, \nu)} 
	\to 0 \quad\hbox{as}~k \to \infty.
\end{align*}
In addition, by \eqref{eq:uktou},
\begin{align*}
\bigg|\int_{\R^n}&(-\Delta)^su_k(x)\varphi(x) \, dx -\frac{1}{\Gamma(-s)} \int_{\R^n} \int_0^{\infty} \( e^{t \Delta} u(x) - u(x)\) \frac{dt}{t^{1+s}}  \, \varphi(x) \, dx \bigg|\\
&\leq C \int_{\R^n} \abs{  (-\Delta)^su_k(x) -\frac{1}{\Gamma(-s)}  \int_0^{\infty} \( e^{t \Delta} u(x) - u(x)\) \frac{dt}{t^{1+s}}  } \,\frac{1}{1 + \abs{x}^{n+2s}} \, dx\\
&\leq C \norm{ (-\Delta)^su_k(x) -\frac{1}{\Gamma(-s)}\int_0^{\infty} \( e^{t \Delta} u(x) - u(x)\) \frac{dt}{t^{1+s}}  }_{L^p(\R^n, \nu)}\to 0
\end{align*}
as $k \to \infty$.
Therefore
\begin{align*}
\int_{\R^n} (-\Delta)^su(x)\, \varphi(x) \, dx
	&= \int_{\R^n} u(x) (-\Delta)^s \varphi(x) \, dx \\
	&= \lim_{k \to \infty}  \int_{\R^n} u_k(x) (-\Delta)^s \varphi(x) \, dx \\
	&= \lim_{k \to \infty} \int_{\R^n} (-\Delta)^s u_k(x) \, \varphi(x) \, dx \\
	&= \int_{\R^n}\bigg[ \frac{1}{\Gamma(-s)} \int_0^{\infty} \( e^{t \Delta} u(x) - u(x)\) \frac{dt}{t^{1+s}} \bigg] \, \varphi(x) \, dx,
\end{align*}
and so we obtain
$$(-\Delta)^su(x) = \frac{1}{\Gamma(-s)} \int_0^{\infty} \( e^{t \Delta} u(x) - u(x)\) \frac{dt}{t^{1+s}} \quad\text{for a.e.}~ x \in \R^n.$$

\medskip

\noindent
\underline{\bf Step 3.} The integral expression in \eqref{eq:laplacian pw} defines a function in $L^p(\R^n, \nu)$ for all $\varepsilon >0$.

\smallskip

For $\varepsilon >0$, define the operator $T_{\varepsilon}$ on $L^p(\R^n, \nu)$ by 
\begin{equation}\label{eq:TeTstar}
T_{\varepsilon}u(x) =c_{n,s} \int_{\abs{x-y}> \varepsilon} \frac{u(x) -u(y)}{\abs{x-y}^{n+2s}}\, dy.
\end{equation}
We claim that $T_{\vep} u(x) \in L^p(\R^n, \nu)$ for all $\vep >0$. Indeed,
for $1 < p < \infty$ this is immediate by Theorem \ref{prop:Mbound}: there exists $C>0$ such that
$$\norm{T_{\varepsilon}u}_{L^p(\R^n, \nu)} 
	\leq C \( \norm{M(D^2u)}_{L^p(\R^n, \nu)} + \norm{Mu}_{L^p(\R^n, \nu)} \) < \infty.$$
For $p=1$, we write
\begin{align*}
T_{\varepsilon}u(x) 
	&= c_{n,s}u(x)  \int_{\abs{x-y}>\varepsilon} \frac{1}{\abs{x-y}^{n+2s}}\, dy + c_{n,s} \int_{\abs{x-y}>\varepsilon} \frac{u(y)}{\abs{x-y}^{n+2s}}\, dy \\
	&= c_{n,s}\frac{C_n\varepsilon^{-2s}}{2s} u(x) + c_{n,s} \int_{\abs{x-y}>\varepsilon} \frac{u(y)}{\abs{x-y}^{n+2s}}\, dy.
\end{align*}
We only need to study the second term above.
By applying Lemma \ref{LemmaDuo} with  $\eta(y) = \chi_{\{\abs{y}\leq \varepsilon\}}(y)+\abs{y}^{-n-2s} \chi_{\{\abs{y}>\varepsilon\}}(y)$
and the $A_1(\R^n)$-condition on $\nu$, we find
\begin{align*}
 \norm{\int_{\abs{x-y}>\varepsilon} \frac{u(y)}{\abs{x-y}^{n+2s}}\, dy}_{L^1(\R^n, \nu)}
 	&\leq \int_{\R^n} \abs{u(y)} \int_{\abs{x-y}>\varepsilon} \frac{\nu(x)}{\abs{x-y}^{n+2s}} \, dx \, dy \\
	&\leq\int_{\R^n}|u(y)|(\nu\ast\eta)(y)\,dy \\
	&\leq C_{n,s,\varepsilon}\int_{\R^n}|u(y)|M\nu(y)\,dy \\
	&\leq C_{n,s,\varepsilon,\nu}\norm{u}_{L^1(\R^n, \nu)} < \infty.
\end{align*}

\medskip

\noindent
\underline{\bf Step 4.} The principal value in  \eqref{eq:laplacian pw} converges in  $L^p(\R^n, \nu)$
to the function $(-\Delta)^su$.

\smallskip

We write the semigroup formula \eqref{eq:laplacian semigroup} as 
\begin{align*}
(-\Delta)^s u(x) 
	 &= \frac{1}{\Gamma(-s)} \int_0^{1} \(\int_{\R^n} W_t(x-y) \(u(y) -u(x)\)\, dy \) \frac{dt}{t^{1+s}}\\
	 &\quad+ \frac{1}{\Gamma(-s)} \int_1^{\infty} \(\int_{\R^n} W_t(x-y) \(u(y) -u(x)\)\, dy \) \frac{dt}{t^{1+s}} \\
	 &= I + II
\end{align*}
and, similarly,
\begin{align*}
c_{n,s} \int_{\abs{x-y}>\varepsilon} &\frac{u(x) - u(y)}{\abs{x-y}^{n+2s}} \, dy\\
	&=\frac{1}{\Gamma(-s)} \int_{\abs{x-y}>\varepsilon}\(u(y) - u(x)\)  \(\int_0^{\infty} W_t(x-y) \, \frac{dt}{t^{1+s}}\)  dy \\
	&= \frac{1}{\Gamma(-s)} \int_0^{1} \int_{\abs{x-y}>\varepsilon} W_t(x-y) \( u(y) - u(x)\) \, dy \, \frac{dt}{t^{1+s}}\\
	&\quad + \frac{1}{\Gamma(-s)} \int_1^{\infty} \int_{\abs{x-y}>\varepsilon} W_t(x-y) \( u(y) - u(x)\) \, dy \, \frac{dt}{t^{1+s}}\\
	&= I_{\varepsilon} + II_{\varepsilon}.
\end{align*}
From Theorem \ref{Prop:semigroup} it follows that
\begin{align*}
&\norm{II - II_{\varepsilon}}_{L^p(\R^n, \nu)}\\
	&\quad= \norm{ \frac{1}{\Gamma(-s)} \int_1^{\infty}  \left[\(\int_{\abs{x-y}<\varepsilon} W_t(x-y) u(y) \, dy \) + u (x)  \int_{\abs{z}< \varepsilon} W_t(z) \, dz\right] \frac{dt}{t^{1+s}} }_{L^p(\R^n, \nu)}\\
	&\quad\leq C\int_1^{\infty} \(\norm{\int_{\abs{x-y}<\varepsilon} W_t(x-y) u(y) \, dy}_{L^p(\R^n, \nu)} + \norm{u}_{L^p(\R^n, \nu)} \int_{\abs{z}< \varepsilon} W_t(z) \, dz \)\frac{dt}{t^{1+s}}
	\to 0 
\end{align*}
as $\varepsilon \to 0^+$.
We next show $\norm{I - I_{\varepsilon}}_{L^p(\R^n, \nu)} \to 0$ as $\varepsilon \to 0^+$ as well to conclude the proof. Indeed,
$$\norm{I - I_{\varepsilon}}_{L^p(\R^n, \nu)}
	= \norm{ \frac{1}{\Gamma(-s)} \int_0^{1} \(\int_{\abs{y}< \varepsilon} W_t(y) \(u(x-y) - u(x)\) \, dy\) \frac{dt}{t^{1+s}}}_{L^p(\R^n, \nu)}$$
By Taylor's Remainder Theorem and \eqref{eq:exponential},
\begin{align*}
\bigg|\int_{\abs{y}< \varepsilon}& W_t(y) \(u(x-y) - u(x)\) \, dy \bigg|\\
	&\leq  \int_{\abs{y}< \varepsilon} W_t(y)|y|^2\( \int_0^1 (1-r) \abs{D^2u(x-ry)} \, dr\) dy\\
	&\leq C t  \int_{\abs{y}< \varepsilon} W_{2t}(y) \( \int_0^1 (1-r) \abs{D^2u(x-ry)} \, dr\) dy\\
	&= Ct \int_0^1 (1-r)  \(\int_{\abs{y}< \varepsilon} W_{2t}(y)\abs{D^2u(x-ry)} \, dy\) dr\\
	&= Ct \int_0^1(1-r) \( \int_{\abs{y}<r\varepsilon} W_{2tr^2}(y) \abs{D^2u(x-y)} \,dy\) dr.
\end{align*}
In particular, since $D^2u\in L^p(\R^n,\nu)$, by Theorem \ref{Prop:semigroup},
\begin{equation}\label{eq:epsilontozero}
\bigg|\int_{\abs{y}< \varepsilon} W_t(y) \(u(x-y) - u(x)\) \, dy \bigg|\to 0 \quad\hbox{as}~\varepsilon \to 0^+
\end{equation}
a.e.~in $\R^n$. We continue estimating by
\begin{align*}
\bigg|\int_{\abs{y}< \varepsilon}& W_t(y) \(u(x-y) - u(x)\) \, dy \bigg|\\
	&\leq Ct \int_0^1(1-r) \( \int_{\R^n} W_{2tr^2}(y) \abs{D^2u(x-y)} \, dy\) dr\\
	&\leq  Ct M({D^2u})(x)\int_0^1(1-r) \, dr= C t  M({D^2u})(x).
\end{align*}
Whence, for $1 < p < \infty$, we have
\begin{align*}
|I-I_\varepsilon|
\leq CM({D^2u})(x) \int_0^1  t \,  \frac{dt}{t^{1+s}}
\leq CM({D^2u})(x) \in L^p(\R^n, \nu)
\end{align*}
where $C>0$ is independent of $\varepsilon$.
Thus, by the Dominated Convergence Theorem and \eqref{eq:epsilontozero},
$\lim_{\varepsilon \to 0^+} \norm{I - I_{\varepsilon}}_{L^p(\R^n, \nu)} =0$.
When $p=1$, by following the computations above and by Theorem \ref{Prop:semigroup}, we get
\begin{align*}
\|I -& I_{\varepsilon}\|_{L^1(\R^n, \nu)} \\
	&\leq C \int_{\R^n} \int_0^1 t \int_0^1(1-r) \( \int_{\abs{y}<\varepsilon} W_{2tr^2}(x-y) \abs{D^2u(y)} \, dy\) dr \frac{dt}{t^{1+s}} \nu(x) \, dx\\
	&=  C \int_0^1 \int_0^1 (1-r) \int_{\abs{y}<\varepsilon} \abs{D^2u(y)} \(\int_{\R^n} W_{2tr^2}(x-y) \nu(x) \, dx\) dy \, dr \, \frac{dt}{t^s}\\
	&\leq C\int_0^1 \int_0^1 (1-r)\( \int_{\abs{y}<\varepsilon} \abs{D^2u(y)}   M\nu(y) dy \) dr \, \frac{dt}{t^s}\\
	&\leq C \int_0^1 \int_0^1 (1-r) \(\int_{\abs{y}<\varepsilon} \abs{D^2u(y)}   \nu(y)  dy \) dr \, \frac{dt}{t^s}\\
	&= C\int_{\abs{y}<\varepsilon} \abs{D^2u(y)}   \nu(y)\,dy \to 0 \quad\hbox{as}~\varepsilon \to 0^+.
\end{align*}

\medskip

\noindent
\underline{\bf Step 5.}  The principal value in  \eqref{eq:laplacian pw} converges almost everywhere in $\R^n$
to $(-\Delta)^su$.

\smallskip

It follows from Theorem \ref{prop:Mbound} and the properties of $M$ that the operator $T^*$ defined by
$$T^*u(t) = \sup_{\varepsilon>0} |T_\varepsilon u(x)|\quad\hbox{for}~u\in W^{2,p}(\R^n,\nu),$$
where $T_\varepsilon$ is defined as in \eqref{eq:TeTstar}, satisfies the estimates
$$\norm{T^*u}_{L^p(\R^n,\nu)} \leq  C \norm{u}_{W^{2,p}(\R^n,\nu)}\quad\hbox{for any}~u\in W^{2,p}(\R^n,\nu),~1 < p < \infty$$
and
$$\nu\big(\{x\in\R^n:|T^*u(x)|>\lambda\}\big)\leq\frac{C}{\lambda}\|u\|_{W^{2,1}(\R^n,\nu)}\quad\hbox{for any}~u\in W^{2,1}(\R^n,\nu),
~\lambda>0$$
where $C>0$ is independent of $u$. In particular, $T^*$ is bounded from $W^{2,p}(\R^n,\nu)$ into weak-$L^p(\R^n,\nu)$, for
any $1\leq p<\infty$. With these estimates, as in Step 5 of the proof of Theorem \ref{thm:derivative1}$(a)$,
we find that the set
$$E= \left\{ u \in W^{2,p}(\R^n, \nu) : \lim_{\varepsilon \to 0^+} T_{\varepsilon} u(x) = (-\Delta)^su(x) \,\,\hbox{a.e.}\right\}$$
is closed in $W^{2,p}(\R^n, \nu)$. Since $C^{\infty}_c(\R^n) \subset E$, by density, we obtain $E= W^{2,p}(\R^n, \nu)$.

\medskip

\noindent\underline{\textbf{Step 6.}}  The limit as $s\to1^-$ in \eqref{eq:limits} holds in $L^p(\R, \nu)$.

\smallskip

Fix $\varepsilon >0$.  By Theorem \ref{Prop:semigroup}, there exists $\delta >0$ such that 
$$
\norm{e^{t\Delta} \Delta u - \Delta u}_{L_p(\R^n, \nu)} < \varepsilon\quad \text{when}~\abs{t} < \delta.
$$
We write
\begin{align*}
(-\Delta)^s u(x) &= \frac{1}{\Gamma(-s)} \int_0^{\delta} \(e^{t \Delta} u(x) - u(x)\) \frac{dt}{t^{1+s}} + \frac{1}{\Gamma(-s)} \int_{\delta}^{\infty} \(e^{t \Delta} u(x) - u(x)\) \frac{dt}{t^{1+s}}\\
& = I_\delta + II_\delta.
\end{align*}
Looking at the second term, by Theorem \ref{Prop:semigroup},
\begin{align*}
\norm{II_\delta}_{L^p(\R^n, \nu)}
	&\leq \frac{1}{\abs{\Gamma(-s)}} \int_{\delta}^{\infty} \(\norm{e^{t \Delta} u}_{L^p(\R^n, \nu)} + \norm{u}_{L^p(\R^n, \nu)}\) \frac{dt}{t^{1+s}}\\
	&\leq \frac{C \norm{u}_{L^p(\R^n, \nu)}}{\abs{\Gamma(-s)}} \int_{\delta}^{\infty} t^{-1-s} \, dt
	= C \norm{u}_{L^p(\R^n, \nu)} \, \delta^{-s} \frac{(1-s)}{\abs{\Gamma(2-s)}} \to 0
\end{align*}
as $s \to 1^-$. Next,
\begin{align*}
\|I_\delta -& (-\Delta)u\|_{L^p(\R^n, \nu)} \\
	&= \norm{\frac{1}{\Gamma(-s)} \int_0^{\delta} \int_0^t \partial_r e^{r\Delta} u(x) dr\,  \frac{dt}{t^{1+s}} +\Delta u(x)}_{L^p(\R^n, \nu)}\\
	&= \norm{\frac{1}{\Gamma(-s)} \int_0^{\delta} \int_0^t e^{r\Delta} \Delta u(x) dr\,  \frac{dt}{t^{1+s}} + \Delta u(x)}_{L^p(\R^n, \nu)}\\
	&=  \norm{\frac{1}{\Gamma(-s)} \int_0^{\delta} \int_0^t \(e^{r\Delta} \Delta u(x) -\Delta u(x)\) dr\,  \frac{dt}{t^{1+s}} + 
	\(\frac{(-s)\, \delta^{1-s}}{\Gamma(2-s)}  +1\)\Delta u(x)}_{L^p(\R^n, \nu)} \\
	&\leq \frac{1}{\abs{\Gamma(-s)}} \int_0^{\delta} \int_0^t \norm{e^{r\Delta} \Delta u  -\Delta u}_{L^p(\R^n, \nu)} dr\,  \frac{dt}{t^{1+s}}
+\abs{\frac{(-s)\, \delta^{1-s}}{\Gamma(2-s)}  +1} \norm{\Delta u}_{L^p(\R^n, \nu)} \\
	&\leq \varepsilon \, \delta^{1-s} \, \frac{s}{\abs{\Gamma(2-s)}} + \abs{\frac{(-s)\, \delta^{1-s}}{\Gamma(2-s)}  +1} \norm{\Delta u}_{L^p(\R^n, \nu)}\to \varepsilon \quad\hbox{as}~s \to 1^-.
\end{align*}
Since $\varepsilon >0$ was arbitrary, \eqref{eq:limits} follows in $L^p(\R^n,\nu)$.

\medskip

\noindent
\noindent\underline{\textbf{Step 7.}}  The limits as $s\to1^-$ in \eqref{eq:limits}
and as $s\to0^+$ in \eqref{eq:limits2} hold a.e. in $\R^n$.

\smallskip

This is proved as in Step 5, by noticing that $\sup_{0<s<1}|(-\Delta)^su(x)|$ can be bounded by means of Theorem \ref{prop:Mbound}
and that
$\lim_{s\to1^-}(-\Delta)^su(x)=-\Delta u(x)$
and $\lim_{s\to0^+}(-\Delta)^su(x)=u(x)$
for all $x\in\R^n$, for any $u \in C^\infty_c(\R^n)$.

\medskip

\noindent\underline{\textbf{Step 8.}} The limit as $s\to0^+$ in \eqref{eq:limits2} holds in $L^p(\R^n, \nu)$.

\smallskip

By Theorem \ref{prop:Mbound}, for any $0<s<1$,
\begin{align*}
\abs{(-\Delta)^s u(x)- u(x)}^p \nu(x) 
	&\leq \(C_n (M(D^2 u)(x) + M u(x)) + \abs{u(x)}\)^p \nu(x) \\
	&\leq C_{n,p} \((M(D^2 u)(x))^p + (M u(x))^p\) \nu(x).
\end{align*}
Therefore, by Step 7 and the Dominated Convergence Theorem, \eqref{eq:limits2} holds in $L^p(\R^n, \nu)$.

\medskip

This completes the proof of Theorem \ref{thm:laplacian1}, part $(a)$.\qed

\subsection{Proof of Theorem \ref{thm:laplacian1}$(b)$}

Suppose $(-\Delta)^s u \to v$ in $L^p(\R^n, \nu)$ as $s\to1^-$. Let $\varphi \in C^{\infty}_c(\R^n)$ and observe that
\begin{align*}
\int_{\R^n} v  \varphi \, dx
	&= \lim_{s \to 1^-} \int_{\R^n} (-\Delta)^su  \varphi \, dx \\
	&= \lim_{s \to 1^-} \int_{\R^n} u  (-\Delta)^s \varphi \, dx\\
	&= \int_{\R^n} u (-\Delta) \varphi \, dx=(-\Delta u)(\varphi).
\end{align*}
In the first line we used that, by Proposition \ref{prop:LsLp} and the fact that $\varphi \in C^\infty_c(\R^n)$,
\begin{align*}
\bigg|\int_{\R^n} v(x) \varphi(x) \, dx - \int_{\R^n} (-\Delta)^s u(x) \varphi(x) \, dx \bigg|
	&\leq \int_{\R^n}\abs{v(x) - (-\Delta)^s u(x)} \frac{C_\varphi}{1+ \abs{x}^{n}} \, dx \\
	&\leq C_{\varphi,n,p,\nu}\norm{v - (-\Delta)^su}_{L^p(\R^n, \nu)} \to 0
\end{align*}
as $s \to 1^-$, while
in the second to last identity we used the Dominated Convergence Theorem, the fact that $(-\Delta)^s\varphi\in\mathcal{S}_s$,
and Proposition \ref{prop:LsLp} in the case of $L_0$.

Therefore, $v = -\Delta u$ a.e.~in $\R^n$. Since $v \in L^p(\R^n,\nu)$, we get that $\Delta u \in L^p(\R^n,\nu)$.
Now we apply the weighted Calder\'on--Zygmund estimates (see \cite{Duo}). Hence, if $1<p<\infty$, then $u \in W^{2,p}(\R^n, \nu)$
and, as a consequence of part $(a)$, \eqref{eq:limits} holds. On the other hand, if $p=1$, then $D^2u\in$ weak-$L^1(\R^n,\nu)$.
\qed

\subsection{Proof of Theorem \ref{thm:laplacian1}$(c)$}

Suppose $(-\Delta)^s u \to v$ in $L^p(\R^n, \nu)$ as $s\to0^+$, and let $\varphi \in C^{\infty}_c(\R^n)$.
Using the exact same arguments as in part $(b)$, we find that
\begin{align*}
\int_{\R^n} v \varphi\,dx
	&= \lim_{s \to 0^+} \int_{\R^n} (-\Delta)^su \varphi \, dx \\
	&= \lim_{s \to 0^+} \int_{\R^n} u (-\Delta)^s \varphi \, dx= \int_{\R^n} u \varphi \, dx.
\end{align*}
Therefore, $u=v=\lim_{s\to0^+}(-\Delta)^su$ a.e.~in $\R^n$ and the result follows. \qed

\medskip

\noindent\textbf{Acknowledgements.} We are grateful to F.~J.~Mart\'in-Reyes and J.~L.~Torrea
for very interesting comments and observations. We also thank the referees
for helping us improve the presentation of our paper. 




\begin{thebibliography}{10}

\bibitem{Allen} M.~Allen,
{H\"older regularity for nondivergence nonlocal parabolic equations},
\textit{Calc. Var. Partial Differential Equations}
\textbf{57} (2018), Art.~110, 29 pp.

\bibitem{Allen-Caffarelli-Vasseur1} M.~Allen, L.~A.~Caffarelli and A.~Vasseur,
{A parabolic problem with a fractional time derivative},
\textit{Arch. Ration. Mech. Anal.}
\textbf{221} (2016), 603--630.

\bibitem{Allen-Caffarelli-Vasseur2} M.~Allen, L.~A.~Caffarelli and A.~Vasseur,
{Porous medium flow with both a fractional potential pressure and fractional time derivative},
\textit{Chin. Ann. Math. Ser. B}
\textbf{38} (2017), 45--82.

\bibitem{Bernardis-et-al} A.~Bernardis, F.~J.~Mart\'in-Reyes, P.~R.~Stinga and J.~L.~Torrea,
{Maximum principles, extension problem and inversion for nonlocal one-sided equations},
\textit{J. Differential Equations}
\textbf{260} (2016), 6333--6362.

\bibitem{Bourgain-Brezis-Mironescu} J.~Bourgain, H.~Brezis and P.~Mironescu,
{Another look at Sobolev spaces},
in: \textit{Optimal Control and Partial Differential Equations},
439--455, IOS, Amsterdam, 2001.

\bibitem{DiMarino-Squassina} S.~Di Marino and M.~Squassina,
{New characterization of Sobolev metric spaces},
arXiv:1803.01658 (2018), 15pp.

\bibitem{Dipierro-Valdinoci} S.~Dipierro and E.~Valdinoci,
{A density property for fractional weighted Sobolev spaces},
\textit{Atti Accad. Naz. Lincei Rend. Lincei Mat. Appl.}
\textbf{26} (2015), 397--422.

\bibitem{Duo} J.~Duoandikoetxea,
\textit{Fourier Analysis},
Graduate Studies in Mathematics \textbf{29},
American Mathematical Society,
Providence, RI, 2001.

\bibitem{Hardy-Littlewood} G.~H.~Hardy and J.~E.~Littlewood,
{A maximal theorem with function-theoretic applications},
\textit{Acta Math.}
\textbf{54} (1930), 81--116.

\bibitem{Lorente} M.~Lorente,
{The convergence in $L^1$ of singular integrals in harmonic and ergodic theory},
\textit{J. Fourier Anal. Appl.}
\textbf{5} (1999), 617--638.

\bibitem{Marchaud} A.~Marchaud,
{Sur les d\'eriv\'ees et sur les diff\'erences des fonctions de variables r\'eelles},
\textit{J. Math. Pures Appl.}
\textbf{6} (1927), 337--425. 

\bibitem{MartinBMO} F.~J.~Mart\'in-Reyes and A.~de la Torre,
{One-sided BMO spaces},
\textit{J. London Math. Soc. (2)}
\textbf{49} (1994), 529--542.

\bibitem{Cabrelli-Torrea} F.~J.~Mart\'in-Reyes, P.~Ortega and A.~de la Torre,
{Weights for one-sided operators},
in: \textit{Recent Developments in Real and Harmonic Analysis}, 97--132,
Appl. Numer. Harmon. Anal.,
Birkh\"auser Boston, Inc.,
Boston, MA, 2010.

\bibitem{Martin-Reyes-et-al} F.~J.~Mart\'in-Reyes, P.~Ortega Salvador and A.~de la Torre,
{Weighted inequalities for one-sided maximal functions},
\textit{Trans. Amer. Math. Soc.}
\textbf{319} (1990), 517--534. 

\bibitem{Canadian} F.~J.~Mart\'in-Reyes, L.~Pick and A.~de la Torre,
$A_\infty^+$ condition,
\textit{Canad. J. Math.}
\textbf{45} (1993), 1231--1244.

\bibitem{Samko} S.~G.~Samko, A.~A.~Kilbas, and O.~I.~Marichev,
\textit{Fractional Integrals and Derivatives: Theory and Applications},
Gordon and Breach, Yverdon, 1993.

\bibitem{Sawyer} E.~Sawyer,
{Weighted inequalities for the one-sided Hardy-Littlewood maximal functions},
\textit{Trans. Amer. Math. Soc.}
\textbf{297} (1986), 53--61.

\bibitem{Silvestre} L.~E.~Silvestre.
{Regularity of the obstacle problem for a fractional power of the Laplace operator},
Thesis (Ph.D.)--The University of Texas at Austin
(2005), 95pp.

\bibitem{Stinga} P.~R.~Stinga,
{User's guide to fractional Laplacians and the method of semigroups},
In Anatoly Kochubei, Yuri Luchko (Eds.), \emph{Fractional Differential Equations}, 235--266,
Berlin, Boston: De Gruyter, 2019. https://doi.org/10.1515/9783110571660-012

\bibitem{Stinga-Torrea} P.~R.~Stinga and J.~L.~Torrea,
{Extension problem and Harnack's inequality for some fractional operators},
\textit{Comm. Partial Differential Equations}
\textbf{35} (2010), 2092--2122.

\bibitem{Turesson} B.~O.~Turesson,
\textit{Nonlinear Potential Theory and Weighted Sobolev Spaces},
Lecture Notes in Mathematics \textbf{1736},
Springer-Verlag, Berlin, 2000.

\bibitem{Vincent} M.~Vincent,
{Integral energy characterization of Hajlasz--Sobolev spaces},
\textit{J. Math. Anal. Appl.}
\textbf{425} (2015), 381--406.

\end{thebibliography}
\end{document}